\documentclass{amsart}
\usepackage{amsmath,amsthm,amsfonts,amssymb,amscd}
\usepackage{framed}
\usepackage{nicefrac}
\usepackage{color,graphicx}
\usepackage[x11names, rgb]{xcolor}
\usepackage{tikz}
\usetikzlibrary{cd}
\usepackage{hyperref}
\usepackage{booktabs}

\usepackage{pdfsync}
\usepackage{enumitem}
\usepackage[capitalize,nameinlink]{cleveref}

\usepackage[utf8]{inputenc}
\usepackage[position=bottom]{subfig}
\usepackage{youngtab}
\usepackage{longtable}

\usepackage{thmtools}
\Crefname{conjecture}{Conjecture}{Conjectures}
\newtheorem{thm}{Theorem}[section]
\newtheorem{cor}[thm]{Corollary}
\newtheorem{prop}[thm]{Proposition}
\newtheorem{lemma}[thm]{Lemma}
\newtheorem{example}[thm]{Example}
\newtheorem{conj}[thm]{Conjecture}
\newtheorem{remark}[thm]{Remark}
\newtheorem{defn}[thm]{Definition}

\DeclareMathOperator{\HS}{HS}

\begin{document}

\title[Difference Formula]{A Difference Formula for Tensor-power Multiplicities}

\author{Hongfei Shu}
\address{Institute for Astrophysics, School of Physics, Zhengzhou University, Zhengzhou, Henan 450001, China}
\email{shu@zzu.edu.cn}

\author{Peng Zhao}
\address{
International Joint Institute of Tianjin University, Fuzhou, 350207, China}
\email{pzhao@tjufz.org.cn}

\author{Rui-Dong Zhu}
\address{Institute for Advanced Study \& School of Physical Science and Technology, Soochow University, Suzhou 215006, China}
\address{Jiangsu Key Laboratory of Frontier Material Physics and Devices, Soochow University, Suzhou 215006, China}
\email{rdzhu@suda.edu.cn}

\author{Hao Zou}
\address{Center for Mathematics and Interdisciplinary Sciences, Fudan University, Shanghai 200433, China}
\address{Shanghai Institute for Mathematics and Interdisciplinary Sciences, Shanghai 200433, China}
\address{Department of Mathematics, The Chinese University of Hong Kong, Shatin, Hong Kong Special Administrative Region}
\email{haozou@fudan.edu.cn}


\begin{abstract} 
A novel combinatorial formula is developed for for tensor product multiplicities in representation theory. We introduce a difference formula linking these multiplicities to restricted occupancy coefficients via a shifted operator. This method is extended to derive branching rules for subalgebras and is conjecturally applied to A-type Lie superalgebras.
\end{abstract}
\maketitle

\setcounter{tocdepth}{1}
\tableofcontents
\section{Introduction and main results}

Let $W$ be a finite-dimensional representation of a complex semisimple Lie algebra $\mathfrak{g}$. One of the basic and widely studied problems in representation theory is the decomposition of tensor powers $W^{\otimes L}$ into irreducible components,
\begin{equation}
W^{\otimes L} \:=\: \bigoplus_{\lambda} \mu_\lambda\, V_{\lambda} \,.
\end{equation}
Here $\lambda$ ranges over the highest weights of all irreducible representations of $\mathfrak{g}$. The multiplicity of irreducible representations of tensor products is a central object in the theory of irreducible representations. These typically appear in the form of the Littlewood–Richardson rule or the character method.
The goal of this paper is to establish a formula for $\mu_\lambda$, establishing a new relationship between combinatorics and representation theory. Our formula expresses the tensor product multiplicities as signed sums of restricted occupancy coefficients, and is reminiscent of Steinberg's formula involving signed sums of the Kostant partition function.

More specifically, we focus on the case $\mathfrak{g}=A_r$ in this paper and take the representation $W$ as $V_{(2s)}$, which is the spin-$s$ irreducible representation of $A_r$ labeled by the one-row Young diagram consisting of $2s$ boxes and is isomorphic to the $(2s)$-symmetric representation ${\rm Sym}^{2s}\left(\mathbb{C}^{r+1}\right)$. Here $s$ is an integer or a half-integer. From a computational perspective, the decomposition of the tensor product representation can be expressed using the Weyl character formula
\begin{equation}
\label{eq:schurexpan}
    \left[\chi_{(2s)}(x)\right]^L \:=\: \sum_{\lambda} \mu_{\lambda} \chi_{\lambda}(x)\,.
\end{equation}
The Weyl character for the irreducible representation labeled by $\lambda=(\lambda_1,\dots,\lambda_{r+1})$ can be expressed by the {\it Schur polynomial} of $x:=(x_1,\dots,x_{r+1})$
\begin{equation}
\label{eq:schurpoly}
   \chi_\lambda (x)\:=\: \frac{\sum_{\omega\in \mathcal{S}_{r+1}}{{\rm sgn}}(\omega)\prod_{i=1}^{r+1}x_{\omega(i)}^{\lambda_i+r+1-i}}{\sum_{\omega\in \mathcal{S}_{r+1}}{{\rm sgn}}(\omega)\prod_{i=1}^{r+1}x_{\omega(i)}^{r+1-i}}\,.
\end{equation}
For simplicity of notation, in the following context without causing any confusion, we would also like to write 
\[
\Delta(x) \:=\: \sum_{\omega\in \mathcal{S}_{r+1}}{{\rm sgn}}(\omega)\prod_{i=1}^{r+1}x_{\omega(i)}^{r+1-i}\:=\: \prod_{1 \leq i<j\leq r+1}(x_i-x_j)\,.
\]
Then, the left hand side of \eqref{eq:schurexpan} is the $L$-th power of $\chi_{(2s)}(x)$, that is the $(2s)$-th {\it Complete Homogeneous Symmetric Polynomial}, which again gives a homogeneous symmetric polynomial of degree $2sL$, while the right hand side is the summation of Schur polynomials $\chi_{\lambda}(x)$ with the multiplicity $\mu_{\lambda}$, over all possible Young diagrams labeling irreducible representations of $A_r$.

\begin{remark}\label{rmk:red}
In this paper, for the purpose of using the equation \eqref{eq:map}, we use $(r+1)$-dimensional highest weights determined by $(\lambda_1, \dots, \lambda_{r+1})$ for irreducible representations of $A_r$, which gives redundant Young diagrams with $(r+1)$ rows. 

The redundancy can be eliminated by using the relation for the standard basis $L_1+L_2 + \dots+ L_{r+1} = 0$ in $A_r$, which implies
$$(\lambda_{1}, \cdots, \lambda_{r+1}) \:\sim\: (\lambda_{1}-\lambda_{r+1}, \cdots, \lambda_{r}-\lambda_{r+1},0)\,.$$
Therefore, if $\lambda_{r+1}\neq 0$, one can reduce the $(r+1)$-row redundant Young diagrams to $r$-row Young diagrams by eliminating the first $\lambda_{r+1}$ columns.
\end{remark}

On the other hand, we consider another expansion of $\left[\chi_{(2s)}\right]^L$ as follows.
\begin{equation}
\label{eq:genexpan}
    \left[\chi_{(2s)}(x)\right]^L \:=\: \sum_{\vec{M}} c_{s,L}(\vec{M})\, x_{1}^{2sL-M_1}x_{2}^{M_1-M_2}\cdots x_{r}^{M_{r-1}-M_r}x_{r+1}^{M_{r}}\,,
\end{equation}
with $\vec{M}:=(M_1,M_2,\dots,M_r)$ and $0\leq M_r \leq \dots \leq M_1 \leq 2sL$. The motivation for this expansion comes from its relation to the enumeration of spin-chain states as discussed in \cite{Shu:2022vpk,Shu:2024crv}, where $\left[\chi_{(2s)}(x)\right]^L$ was interpreted as the generating function for a 3d restricted occupancy problem \footnote{The authors would like to thank Arkadij Bojko for pointing out to us that this is also related to the colored instanton counting problem.}, as summarized in Appendix~\ref{app:nrop}. The coefficients $c_{s,L}(\vec{M})$ also give the number of Physical Bethe states in the XXX spin chain with the most generic twist boundary condition (see \cite{Shu:2022vpk,Shu:2024crv} for more detailed discussions).

As motivated in \cite{Shu:2022vpk,Shu:2024crv}, we would like to establish the dictionary between $\lambda$ and $\vec{M}$ by
\begin{equation}
\label{eq:map}
    \lambda_i \:=\: M_{i-1} - M_i,\quad \text{for } i=1,\dots,r+1\,,
\end{equation}
and we have adopted the convention that $M_0 = 2sL$ and $M_{r+1}=0$. 
\begin{remark}\label{rmk:order}
One may note that the standard Young diagram $\lambda$ by definition requires $\lambda_1 \geq \lambda_2 \geq \cdots \geq \lambda_{r+1} \geq 0$, while the relation among $M_i$'s, $2sL\geq M_1 \geq \cdots \geq M_{1} \geq 0$, does not imply $M_{i}-M_{i+1} \leq M_{i-1} -M_{i} $, which seems to contradict with \eqref{eq:map}. However, we shall mention that we will only need to look at ranges of $M_i$'s such that $M_{i}-M_{i+1} \leq M_{i-1} -M_{i}$, namely $(2sL-M_1,M_1-M_2,\dots,M_r-M_{r+1},M_r)$ becomes a standard partition, and for $M_i$ beyond these ranges will be disregarded due to the symmetry property, which ensures that they will automatically self-organized.
\end{remark}
The first result of this paper provides a formula that relates the coefficients $\mu_{\lambda}$ and $c_{s,L}(\vec{M})$, which we call the {\it difference formula},
\begin{thm}\label{thm:main}
The multiplicity $\mu_\lambda$ of the occurrence of the irreducible representation $V_\lambda$ of the Lie algebra $A_r$ in the tensor power $V_{(2s)}^{\otimes L}$ is given by the formula
\begin{equation}
\mu_\lambda  \:=\: \mathcal{D}_{R_{A_r}(t)}  c_{s,L}(\vec{M}) \,,
\label{eq:thm}
\end{equation}
where $\mathcal{D}_{R_{A_r}(t)}$ is the shifted operator associated with the Weyl denominator of $A_r$, $R_{A_r}(t)$, as defined in \eqref{eq:weyldenom}, and the relation of $\lambda$ and $\vec{M}$ is given by \eqref{eq:map}.
\end{thm}
This theorem follows immediately after the proof of Theorem \ref{thm:main1} and the introduction of the shifted operator in Section~\ref{sec:proof}.

Furthermore, motivated by the twisted boundary conditions of the XXX spin chain considered in \cite{Shu:2024crv}, the global symmetry Lie algebra $\mathfrak{g}$ will break down to its subalgebra $\mathfrak{h}$, leading to the consideration of the decomposition of the tensor product into its restricted representations of $\mathfrak{h}$:
\begin{equation}
\label{eq:branching}
\left(V^{\mathfrak{g}}_{(2s)}\right)^{\otimes L} \:=\: \bigoplus_{\lambda} \mu^{\mathfrak{h}}_\lambda V^{\mathfrak{h}}_{\lambda} \,.
\end{equation}
In the general case, the Lie subalgebra $\mathfrak{h}$ takes the form of a direct sum
\begin{equation}
    \mathfrak{h} \:=\: \mathfrak{g}_1 \oplus \cdots \oplus \mathfrak{g}_k\,,
\end{equation}
therefore, the summation in Equation~\eqref{eq:branching} is over the irreducible representations for this direct sum Lie algebra. In other words, the $\lambda$ in Equation~\eqref{eq:branching} should be expressed as $k$-components and the $i$-th component labels an irreducible representation of $\mathfrak{g}_i \subset \mathfrak{g}$. 

The coefficients $\mu^{\mathfrak{h}}_\lambda$ are the decomposition multiplicities counting the occurrence of restricted representations, and we have used a superscript $\mathfrak{h}$ to address the relevant subalgebra. From this viewpoint, the expansion \eqref{eq:genexpan} can be understood as a special case when $\mathfrak{h} = \mathfrak{gl}(1)^{r}$ and the coefficients $\mu^{\mathfrak{gl}(1)^{r}}_{\lambda} = c_{s,L}(\vec{M})$, while the expansion \eqref{eq:schurexpan} is another extremal case when $\mathfrak{h} = \mathfrak{g}$ and $\mu^{\mathfrak{g}}_{\lambda} = \mu_{\lambda}$.

The second result of this paper is to generalize Theorem~\ref{thm:main} and provides an explicit formula relating the coefficients $c_{s,L}(\vec{M})$ and $\mu^{\mathfrak{h}}_\lambda$, which will be proven in Section~\ref{sec:branch}:
\begin{thm}
\label{thm:branching}
Consider the restriction of the tensor production representation of $\mathfrak{g}=A_r$ onto its subalgebra $\mathfrak{h}=\mathfrak{g}_1 \oplus \cdots \oplus \mathfrak{g}_k$, the multiplicity $\mu^{\mathfrak{h}}_\lambda$ can be obtained by
\begin{equation}
\label{eq:branchingmulti}
    \mu^{\mathfrak{h}}_\lambda \:=\: \mathcal{D}_{R_{\mathfrak{h}}(t)} c_{s,L}(\vec{M})\, ,
\end{equation}
where $\mathcal{D}_{R_{\mathfrak{h}}(t)}$ is the shift operator associated to the Weyl denominator of subalgebra $\mathfrak{h}$, which is the product of Weyl denominator of each $g_i$ for $i=1,\dots,k$. Here, the relation between $\lambda$ and $\vec{M}$ is again given by Equation~\eqref{eq:map}.
\end{thm}

We further extend the above formulas to the $A$-type Lie superalgebra $\mathfrak{sl}(m|n)$ and consider the tensor power decomposition,
\[
    V_{(2s)}^{\otimes L} \:=\: \oplus_{\lambda}\, \mu_{\lambda}\, V_{\lambda}\,,
\]
where $V_\lambda$ are the covariant tensor representations $V_{\lambda}$ of $\mathfrak{sl}(m|n)$ with $\lambda$ being $(m,n)$-hook Young diagrams, namely constrained in the $(m,n)$-hook as shown in Figure~\ref{fig:hook}. Characters of $V_{\lambda}$ can be represented by hook-Schur polynomials $\HS_{\lambda}(x;y)$ \cite{Berele:1987yi,VanderJeugt:1989ak}. Then, as before, we consider the following two expansions of $L$-th power of $V_{(2s)}$, written in terms of hook-Schur polynomials,
\begin{equation}
    \begin{aligned}
        \left[\HS_{(2s)}(x;y)\right]^L &\:=\: \sum_{\lambda} \mu_{\lambda} \, \HS_{\lambda}(x;y)\,, \\
        &\:=\: \sum_{\vec{M}} c_{s,L}(\vec{M})\ \prod_{i=1}^m x_i^{M_{i-1}-M_i} \prod_{\ell=1}^n y_\ell^{M_{m+\ell-1}-M_{m+\ell}}\,.
    \end{aligned}
\end{equation}
For simplicity, we have used the same notation for the expansion coefficients as the ordinary Lie algebras, but the intended meaning should be clear from the context.

The dictionary between $\lambda$ and $\vec{M}$ should be modified because now the Young diagrams should be constrained inside the $(m,n)$-hook. It turns out that we need to separate the $(m,n)$-hook Young diagram $\lambda$ into two parts: the first $m$ rows of $\lambda$, which is a smaller standard Young diagram $(\lambda_1,\dots,\lambda_m)$; and the remaining part $\tilde{\lambda} = \lambda / (\lambda_1,\dots,\lambda_m)$, which consists of $n$ columns. Then, the generalization of \eqref{eq:map} to the case of $\mathfrak{sl}(m|n)$ is given by
\begin{equation}\label{eq:mapsuper}
    \lambda_i \:=\: M_{i-1} - M_i,\quad  {\tilde{\lambda}^{\prime}}_\ell \:=\: M_{\ell+m-1} - M_{\ell+m}\,,
\end{equation}
for $i=1,\dots,m$ and $\ell=1,\dots,n$ with the convention $M_0 = 2sL$ and $M_{m+n} =0$, and $\tilde{\lambda}^\prime$ is the conjugate Young diagram of $\tilde{\lambda}$ defined above. Note that Remark~\ref{rmk:order} also applies to \eqref{eq:mapsuper}. Then, under the relation between $\lambda$ and $M_i$'s \eqref{eq:mapsuper}, we conjecture that the coefficients $\mu_{\lambda}$ and $c_{s,L}(\vec{M})$ are related by the following formula:
\begin{conj} 
\label{conj:super}
    \[
    \mu_{\lambda} \:=\: \mathcal{D}_{R_{\mathfrak{sl}(m|n)}(t)} c_{s,L}(\vec{M})\,,
    \]
    where $\mathcal{D}_{R_{\mathfrak{sl}(m|n)}(t)}$ is the shifted operator associated with $R_{\mathfrak{sl}(m|n)}(t)$ defined in \eqref{eq:superdeno}.
\end{conj}

Moreover, let us consider the decomposition of $[V_{(2s)}]^L$ over $V_{\lambda}^{\mathfrak{h}}$ of a subalgebra $\mathfrak{h}$ of $\mathfrak{sl}(m|n)$, and suppose that $\mathfrak{h}$ is determined by a subset of positive roots of $\mathfrak{sl}(m|n)$, $\Psi^{+}$. Then, the expansion coefficients $\mu_{\lambda}^{\mathfrak{h}(\Psi^+)}$ in the decomposition
\begin{equation}\label{eq:superbranching}
        V_{(2s)}^{\otimes L} \:=\: \oplus_{\lambda}\, \mu_{\lambda}^{\mathfrak{h}(\Psi^+)} \, V_{\lambda}^{\mathfrak{h}}
\end{equation}
can also be computed from $c_{s,L}(\vec{M})$ by
\begin{conj} \label{conj:superbranching}
    \[
        \mu_\lambda^{\mathfrak{h}(\Psi^+)} \:=\: \mathcal{D}_{R_{\mathfrak{h}(\Psi^+)}(t)} c_{s,L}(\vec{M})\,,
    \]
    where $R_{\mathfrak{h}(\Psi^+)}(t)$ is given by \eqref{eq:superdenosub}.
    \end{conj}

The proofs of Conjectures~\ref{conj:super} and~\ref{conj:superbranching} are left for future work and the interested readers, but some evidence is provided by proofs in special cases and some numerical investigations in Section~\ref{sec:super}.

\vspace{5mm}

\noindent\thanks{{\textbf{Acknowledgment.} The authors would like to thank Arkadij Bojko for useful discussions. HZ would like to thank the hospitality of the Department of Mathematics and the Institute of Mathematical Sciences at the Chinese University of Hong Kong, where part of this work was done during his visit. H.S. was supported by the National Natural
Science Foundation of China No.12405087, Henan Postdoc Foundation (Grant No.22120055) and the Startup Funding of Zhengzhou University
(Grant No.~121-35220049,~121-35220581). R.Z. was partially supported by the National Natural Science Foundation of China (Grant No.~12105198). H.Z. was partially supported by the National Natural Science Foundation of China (Grant No.~12405083,~12475005) and the Shanghai Pujiang Program (Grant No.~24PJA119).}

\vspace{8mm}

\section{The difference formula and generalizations}
\label{sec:proof}

In this section, we will present and prove the difference formula which relates the expansion coefficients $\mu_\lambda$ and $c_{s,L}(\vec{M})$. Using the shift operator presentation, this difference formula will reproduce Equation~(\ref{eq:thm}) and thus provides a proof for Theorem~\ref{thm:main}.

\subsection{The difference formula} 
Following the convention in the introduction, we want to first present the following difference formula, which expresses $\mu_{\lambda}$ as an ``alternating'' sum of $c_{s,L}(\vec{M})$:
\begin{thm}\label{thm:main1} 
\begin{equation}
\label{eq:thm2}
    \mu_{\lambda} \:=\: \sum_{\sigma\in \mathcal{S}_{r+1}} {\rm sgn}(\sigma) c_{s,L}(M_j - \sum_{i=1}^{j}\left(\sigma(i) - i\right)) \, ,
\end{equation}
where $\lambda$ and $\vec{M}$ is related by Equation~\eqref{eq:map}.
\end{thm}
\begin{remark} This formula reminisce the formula for the tensor product multiplicities (see for example \cite[Equation~(13.202)]{DiFrancesco:1997nk})
\[
    \mathcal{N}_{\lambda \mu} \:=\: \sum_{\omega \in W} {\rm sgn}(\omega) {\rm mult}_{\mu}(\omega\cdot\nu - \lambda)\,,
\]
where the tensor product multiplicities $\mathcal{N}_{\lambda \mu}$ appears as coefficients in
\[
    \chi_\lambda \cdot \chi_\mu \:=\: \sum_{\nu} \mathcal{N}_{\lambda \mu} \chi_{\nu}\,,
\]
and $W$ is the Weyl group and ${\rm mult}_{\mu}(\omega\cdot\nu - \lambda)$ comes from the formal definition of a character
\[
    \chi_\mu \:=\: \sum_{\mu^{\prime}\in \Omega_{\mu}} {\rm mult}_{\mu}(\mu^\prime) e^{\mu^\prime}\,.
\]
One should note the crucial difference between the coefficient ${\rm mult}_{\mu}(\mu^\prime)$ and our coefficient $c_{s,L}(\vec{M})$ from their definitions.
\end{remark}

\begin{example}
    In the case of $\mathfrak{g}=A_1$, namely, $r=1$, one can immediately see that Equation~\eqref{eq:thm2} implies the following relation
    \begin{equation}
    \label{eq:a1thm}
        \mu_{\lambda} \:=\: c_{s,L}(M) - c_{s,L}(M-1)\,,
    \end{equation}
    where $\lambda_1=2sL-M$ and $ \lambda_2 = M$. A brute-force derivation of Equation~\eqref{eq:a1thm} is attached in Appendix~\ref{app:a1}.
\end{example}

One may note that, as a homogeneous symmetric polynomial, $\left[\chi_{(2s)}(x)\right]^L$ can also be expanded in other bases, other than Schur polynomials as in \eqref{eq:schurexpan},
such as {\it monomial symmetric polynomials}. Recall that a monomial symmetric polynomial associated with $\lambda = (\lambda_1,\lambda_2,\dots,\lambda_{r+1})$ is defined as
\begin{equation}
\label{eq:monopoly}
\mathcal{M}_{\lambda}(x) \:=\: \prod_{i=1}^{r+1}x_i^{\lambda_i} + \text{distinct permutations}\,.
\end{equation}

Let us expand $\left[\chi_{(2s)}(x)\right]^L$ in monomial symmetric polynomials and obtain
\begin{equation}
\label{eq:monexpan}
    \left[\chi_{(2s)}(x)\right]^L \:=\: \sum_{\lambda} d_{s,L}(\lambda)\ \mathcal{M}_{\lambda}(x)\,,
\end{equation}
where $d_{s,L}(\lambda)$ denotes the expansion coefficient. Using symmetric properties of $\left[\chi_{(2s)}(x)\right]^L$, one can show that:


\begin{lemma}
\label{lem:equal}
    Given the relation \eqref{eq:map} , one shall have
\begin{equation}\label{eq:lemequal}
d_{s,L}(\lambda) \:=\: c_{s,L}(\vec{M})\,.
\end{equation}
\end{lemma}
\begin{proof}
One can start with expansion \eqref{eq:genexpan} and collect all monomials with coefficients equal to $c_{s,L}(\vec{M})$, using the identity \eqref{eq:identity2}. Then these monomials constitute nothing but the monomial symmetric polynomial $\mathcal{M}_{\lambda}$ with the identification \eqref{eq:map}: 
\[
    \lambda_i \:=\: M_{i-1} - M_{i}\,.
\]
Therefore, the coefficient $d_{s,L}(\lambda)$ in the expansion \eqref{eq:monexpan} must be equal to $c_{s,L}(\vec{M})$ due to the fact that monomial symmetric polynomials form a basis for symmetric polynomials.
\end{proof}

\begin{remark}
    Although $d_{s,L}(\lambda) = c_{s,L}(\vec{M})$, there is a computational advantage in using $c_{s,L}(\vec{M})$ in Theorem~\ref{thm:main} and~\ref{thm:main1}, since $c_{s,L}(\vec{M})$ can be extracted directly from the generating function as discussed in Appendix~\ref{app:nrop}.
\end{remark}

\begin{proof}[Proof of Theorem~\ref{thm:main1}]
Let us first derive the relation between $d_{s,L}(\lambda)$ and $\mu_{\lambda}$. Following a standard trick as in \cite{fulton1991representation}, the actual goal is to compute the coefficient 
\begin{equation}
\label{eq:arcoef}
    \mu_{\lambda}\:=\:\left.\Delta(x)\cdot \left[\chi_{(2s)}(x)\right]^L\right|_{(\lambda_1 + r, \lambda_2 + r-1,\dots,\lambda_r + 1, \lambda_{r+1})}\,,
\end{equation}
namely the factor before $x_1^{\lambda_1+r}x_2^{\lambda_1+r-1}\cdots x_{r+1}^{\lambda_{r+1}}$. 
In the above, the Vandermonde determinant $\Delta(x)$ can be expanded as
\begin{equation}
\label{eq:vandermondeexpand}
    \Delta(x) \:=\: \sum_{\sigma\in \mathcal{S}_{r+1}} {\rm sgn}(\sigma) \prod_{i=1}^{r+1} x_{i}^{r+1-\sigma(i)}\,.
\end{equation}
It is also known that the product $\left[\chi_{(2s)}(x)\right]^L$ can be expanded in terms of monomial symmetric polynomials as \eqref{eq:monexpan}.

Due to the symmetric property of $\left[\chi_{(2s)}(x)\right]^L$, one can extract the multiplicity coefficient $\mu_\lambda$ by looking at one item $\prod_{i=1}^{r+1}x_i^{\lambda_i}$ in $\mathcal{M}_{\lambda}$ and pairing coefficients in \eqref{eq:monexpan} and \eqref{eq:vandermondeexpand}, and then read
\begin{equation} \label{eq:proofdiff2}
\begin{aligned}
    \mu_{\lambda}&\:=\:\left.\Delta(x)\cdot \left[\chi_{(2s)}(x)\right]^L\right|_{(\lambda_1 + r, \lambda_2 + r-1,\dots,\lambda_r + 1, \lambda_{r+1})} \\
    &\:=\: \sum_{\sigma\in \mathcal{S}_{r+1}} {\rm sgn}(\sigma) d_{s,L}(\lambda_j + \sigma(j)-j)\,,
\end{aligned}
\end{equation}
By Lemma~\ref{lem:equal}, we can substitute $d_{s,L}(\lambda)$ by $c_{s,L}(\vec{M})$. The shift $\lambda_j \rightarrow \lambda_j + \sigma(j)-j$ in the above equation can be written in $M_i$'s by \eqref{eq:map} and obtain $M_j \rightarrow M_j-\sum_{i=1}^{j}(\sigma(i)-i)$. Therefore, we can conclude
\[
\mu_{\lambda} \:=\: \sum_{\sigma\in\mathcal{S}_{r+1}} {\rm sgn}(\sigma) c_{s,L}(M_j-\sum_{i=1}^{j}(\sigma(i)-i))\,.
\]
\end{proof}

\subsection{The difference formula in shift operator}
\label{sec:shift}
The difference formula \eqref{eq:thm2} in Theorem~\ref{thm:main1} can be simplified if we rewrite it in terms of {\it shift operator}. Given a function $f(x)$, one can define a shift operator acting on $f(x)$ as below \footnote{This is slightly different from the usual shift operators, which will shift variables by $+1$.}
\[
\hat{t}:=\exp(-\partial_{x}),
\]
which shifts $x$ by $-1$ in $f(x)$, namely, $$\hat{t}: f(x)\mapsto f(x-1).$$
Acting $\hat{t}$ on $f(x)$ consecutively $n$ times leads to $f(x-n)$ and we shall define
\[
{\hat{t}}^n:f(x)\mapsto f(x-n).
\]

The above definition can be generalized to the case of multiple variables. Suppose that given a function $f(x_1, \dots, x_k)$ of $k$ variables, one can define the $k$ shift operators as follows 
\[
\hat{t}_i\::=\: \exp(-\partial_{x_i}), \ i=1,\dots, k,
\]
which shifts each $x_i$ in $f(x_1, \dots, x_k)$ by $-1$ ,
\[
    \hat{t}_i f(x_1, \dots, x_k) \:=\: f(x_1,\dots, x_i - 1, \dots, x_k)\,.
\]
One can further define a shift operator associated with a polynomial based on the above definition as follows.
\begin{defn}
\label{def:shiftoperator}
The shift operator $\mathcal{D}_{P(t_1,\ldots,t_k)}$ associated with the polynomial $P(t_1,\ldots,t_k)$ is defined as $\mathcal{D}_{P(t_1,\ldots,t_k)}:= P(\hat{t}_1,\dots,\hat{t}_k)$. Suppose $P(t_1,\ldots,t_k)$ has the expansion
\[
P(t_1,\ldots,t_k) \:=\: \sum_{\beta} {\rm sgn}(\beta) c_{\beta} {t_1}^{\beta_1} \cdots {t_{k}}^{\beta_k}\,,
\] 
where $\beta =(\beta_1,\dots,\beta_k)$, $c_{\beta}$ are the expansion coefficients, and we write ${\rm sgn}(\beta)$ explicitly for later convenience. Then, the associated shift operator $\mathcal{D}_{P(t_1,\ldots,t_k)}$ acting on $f(x_1,\dots,x_k)$ leads to
\[
\mathcal{D}_{P(t_1,\ldots,t_k)} f(x_1,\dots,x_k) \::=\:  \sum_{\beta} {\rm sgn}(\beta) c_{\beta} f(x_1 - \beta_1, \cdots, x_k - \beta_k) \,.
\]
\end{defn}

\begin{lemma}\label{lem:factor} 
Suppose that a polynomial can be factorized into a product of two polynomials of distinct variables $P(t_1,\dots,t_k) = P_1 (t_1,\dots,t_\ell) P_2 (t_{\ell+1},\dots,t_{k})$. Then, the shift operator associated with $P(t_1,\dots,t_k)$ also factorizes,
\[
    \mathcal{D}_{P(t_1,\dots,t_k)} \:=\: \mathcal{D}_{P_1(t_1,\dots,t_\ell)} \mathcal{D}_{P_2(t_{\ell+1},\dots,t_k)}\,.
\]
\end{lemma}
\begin{proof}
    This result follows immediately from the definition.
\end{proof}

For the general root system of the Lie algebra $\mathfrak{g}$ of rank $r$, let $\{\alpha_i, i=1,\dots,r\}$ be the simple roots and denote by $\Phi^+$ the set of positive roots. A typical positive root can be written as $\alpha = \sum_{i=1}^r n_i \alpha_i$ for $n_i \geq 0$ and not all $n_i$ being zero. 

The Weyl denominator is defined as
\begin{equation}
    R_\mathfrak{g} \:=\: \prod_{\alpha\in \Phi^+}\left(1 -  e^{-\alpha}\right)\,.
\end{equation}
One can expand the above expression using the {\it Weyl denominator identity}
\begin{equation}
    R_\mathfrak{g} \:=\: \sum_{\omega\in W(\mathfrak{g})} {\rm sgn}(\omega) e^{\omega(\rho) - \rho}\,,
\end{equation}
where $\rho = \frac{1}{2}\sum_{\alpha \in \Phi^+} \alpha$ is the Weyl vector and $W(\mathfrak{g})$ is the Weyl group of Lie algebra $\mathfrak{g}$. For later purposes, let us introduce the $r$ formal variables associated with each simple root $t_i:= e^{-\alpha_i}$. Then the Weyl denominator $R_\mathfrak{g}$ can be written in terms of $t$-variables as below,
\begin{equation}\label{eq:weyldenom}
    R_{\mathfrak{g}}(t) \:=\: \prod_{\alpha \in \Phi^+} (1- t^\alpha) \,,
\end{equation}
where we have used the multi-index notation $t^\alpha := t_1^{n_1}\cdots t_r^{n_r}$ for $\alpha = \sum_{i=1}^r n_i \alpha_i$. Further, suppose that the Weyl vector can be written in the fundamental basis as
\begin{equation}
    \rho \:=\: \sum_i \rho_i \alpha_i\,,
\end{equation}
the Weyl denominator identity leads to the following expansion
\begin{equation}\label{eq:weyldenom1}
    R_{\mathfrak{g}}(t) \:=\: \sum_{\omega \in W(\mathfrak{g})} {\rm sgn }(\omega) \prod_{i=1}^{r}{t_i}^{\omega(\rho_i)-\rho_i}\,,
\end{equation}
where $\omega(\rho_i)$ is the coefficient of $\omega(\rho) = \sum_{i=1}^{r} \omega(\rho_i) \alpha_i$.

Instead of considering the generic case, let us focus on the Lie algebra $A_r$ and we can obtain the following lemma:
\begin{lemma}
\label{lem:arweylid}
For $\mathfrak{g}=A_r$, the Weyl denominator \eqref{eq:weyldenom} can be expanded as below
\begin{equation}
\label{eq:Arweylid}
    R_{A_{r}}(t)\:=\:\prod_{\alpha \in \Phi^+}(1-t^\alpha) = \sum_{\omega \in \mathcal{S}_{r+1}} {\rm sgn}(\omega) t_1^{\omega(1)-1} t_2^{\sum_{i=1}^{2}\omega(i)-i} \cdots t_r^{\sum_{i=1}^{r}\omega(i)-i}\,.
\end{equation}
\end{lemma}
\begin{proof} 
The above equation is just a rewrite of the Weyl denominator identity in terms of $t$ variables for the $A_r$ case. It is known that the Weyl denominator identity (or the Vandermonde determinant) gives 
\begin{equation}
\label{eq:vandet}
    \prod_{0 \leq i< j\leq r+1}(x_i - x_j) \:=\: \sum_{\omega \in S_{r+1}} {\rm sgn}(\omega) \prod_{i=1}^{r+1}x^{(r+1)-i}_{\omega(i)}\,,
\end{equation}
where $x_i:=e^{L_i}$ with $L_i$ is the standard basis. The relation between $t$ variables and $x$ variables can be obtained by their definitions,
\begin{equation}
\label{eq:t-x}
    t_1 t_2 \cdots t_j \:=\: e^{L_{j+1}-L_1} \:=\: \frac{x_{j+1}}{x_1}\,.
\end{equation}
Dividing both sides of \eqref{eq:vandet} by a factor of ${x_1}^{r} {x_2}^{r-1} \cdots x_r$ gives \eqref{eq:Arweylid}, which completes the proof.
\end{proof}

\begin{prop} \label{prop:poly} For the case $\mathfrak{g} = A_r$, $R_{A_r}(t) \in \mathbb{Z}[t_1,\dots,t_r]$.
\end{prop}
\begin{proof}
    It is obvious that $R_{A_r}(t)$ is integer-valued, as it is defined as a product of integer-valued expressions. To verify that $R_{A_r}(t)$ is a polynomial, one needs to show $\sum_{i=1}^j \left(\omega(i) - i\right) \in \mathbb{Z}_{\geq 0}$ for all $j=1,\dots,r$ and any $\omega \in \mathcal{S}_{r+1}$. It is obvious that $\sum_{i=1}^j \left(\omega(i) - i\right)$ is an integer as it is a sum of integers. When $\omega \in \mathcal{S}_{j} \times \mathcal{S}_{r+1-j} \leq \mathcal{S}_{r+1}$, where $\mathcal{S}_{j}$ only permute the first $j$ indices, $\sum_{i=1}^j \left(\omega(i) - i\right) = 0$. When $\omega \in \mathcal{S}_{r+1}$ but $\omega \notin \mathcal{S}_{j} \times \mathcal{S}_{r+1-j}$, one can conclude $\sum_{i=1}^j \left(\omega(i) - i\right)>0$. In addition, there are only finite terms in $R_{A_r}(t)$ due to the finite rank of the Weyl group.
\end{proof}


By Lemma~\ref{lem:arweylid} and Proposition~\ref{prop:poly}, we can rewrite the difference formula \eqref{eq:thm2} by using the shift operator $\mathcal{D}_{R_{A_r}(t)}= R_{A_r}(\hat{t})$, acting on the function $c_{s,L}(\vec{M})$, and this will finally lead to our Theorem~\ref{thm:main}. 

Let us end this section by giving an example in the following that illustrates the action of the shift operator associated with $R_{\mathfrak{sl}(3)}(t)$.
\begin{example}
For $\mathfrak{g} = \mathfrak{sl}(3)$, the positive roots are $\alpha_1, \alpha_2, \alpha_1 + \alpha_2$, so 
\[
\begin{aligned}
R_{\mathfrak{sl}(3)}(t) &\:=\: (1-t_1)(1-t_2)(1-t_1t_2)\,,\\
&\:=\:1 - t_1 - t_2 +{t_1}^2 t_2 + t_1 {t_2}^2 - {t_1}^2 {t_2}^2 \,.
\end{aligned}
\]
Then, the associated shift operator acting of a function $f(M_1,M_2)$ gives
\begin{multline*}
\mathcal{D}_{R_{\mathfrak{sl}(3)}(t)} f(M_1, M_2) \:=\: f(M_1,M_2)-f(M_1-1,M_2)-f(M_1,M_2-1)\\
+f(M_1-2,M_2-1)+f(M_1-1,M_2-2)-f(M_1-2,M_2-2)\,.
\end{multline*}
\end{example}

\subsection{Tensor product of different spins}\label{sec:diffspin} The above discussion is for the $L$-th tensor power of the same representation, namely the spin-$s$ representation of $\mathfrak{sl}(r+1)$. One can generalize this to tensor products of different spin representations straightforwardly. Suppose that there are $L$ spins, denoted by $s_i$ for $i=1,\dots,L$, or by $\vec{s}$ for short. Let $|\vec{s}|:= \sum_{i=1}^L s_i$. Consider the following two expansions for the tensor product representation, in terms of Weyl characters,
\begin{equation}
\label{eq:genexpan2}
    \prod_{i=1}^L \chi_{(2s_i)}(x) \:=\: \sum_{\vec{M}} c_{\vec{s}}(\vec{M}) \prod_{j=1}^{r+1} {x_{i}}^{M_{i-1}-M_i} \:=\: \sum_{\lambda} \mu_\lambda \chi_{\lambda}(x)\,,
\end{equation}
where we have used the conventions $M_{0} := 2|\vec{s}| = 2\sum_{i}s_i$ and $M_{r+1} :=0$. Here, the coefficient $c_{\vec{s}}(\vec{M})$ corresponds accordingly to the number of configurations in a nested restricted occupancy problem on a rectangle with one zigzag edge.

One can also expand $\prod_{i=1}^L \chi_{2s_i}(x)$ into monomial symmetric polynomials as \eqref{eq:monexpan},
\[
    \prod_{i=1}^L \chi_{(2s_i)}(x) \:=\: \sum_{\lambda} d_{\vec{s}}(\lambda) \mathcal{M}_{\lambda}(x)\,,
\]
and similarly obtain, for $\lambda$ and $\vec{M}$ related by Equation~\eqref{eq:map},
\[
d_{\vec{s}}(\lambda) \:=\: c_{\vec{s}}(\vec{M})\,.
\]
Following the proof of Theorem~\ref{thm:main} and~\ref{thm:main1}, one can also conclude
\begin{cor}\label{cor:diffspins} For the partition $\lambda$ and $\vec{M}$ related by
\begin{equation}\label{eq:map2}
\lambda_i \:=\: M_{i-1} - M_i
\end{equation}
for $i=1,\dots r+1$, where $M_0 = 2|\vec{s}|$ and $M_{r+1} = 0$, one have
    \[
        \mu_{\lambda} \:=\: \sum_{\sigma\in \mathcal{S}_{r+1}} {\rm sgn}(\sigma) c_{\vec{s}}(M_j - \sum_{i=1}^{j}\left(\sigma(i) - i\right)) \,,
    \]
    or, one can write in terms of the shift operator associated with $R_{A_r}(t)$,
    \[
        \mu_{\lambda} \:=\: \mathcal{D}_{R_{A_r}(t)} c_{\vec{s}}(\vec{M})\,.
    \]
\end{cor}

\vspace{8mm}

\section{Discussions on the difference formula}

\subsection{Pieri's formula} Pieri's formula gives a combinatorial rule for the decomposition of the tensor product of an irreducible representation with the symmetric power of a standard representation. Using Weyl characters, Pieri's formula can be explicitly written as
\begin{equation}
\label{eq:pieri}
    \chi_{\nu}(x) \cdot \chi_{(2s)}(x) \:=\: \sum_{\lambda} \chi_{\lambda}(x)\,,
\end{equation}
where the sum is over the Young diagrams obtained following Pieri's rule, that is to add $2s$ boxes to the Young diagram $\lambda$ such that any two of these added $2s$ boxes are not in the same column.

Let us first apply the Pieri's formula \eqref{eq:pieri} to the case when $\nu = (2s)$, and we shall have
\begin{equation}
\label{eq:square}
    \left[\chi_{(2s)}(x)\right]^2 \:=\: \sum_{\lambda} \chi_{\lambda}(x)\,,
\end{equation}
and there are $(2s+1)$ Young diagrams in the summation. Furthermore, Pieri's formula tells us that the multiplicity coefficients $\mu_{\lambda}=1$ for all possible $\lambda$ in Figure~\ref{fig:square}.
\begin{figure}[!h]
\centering
\includegraphics[scale=0.75]{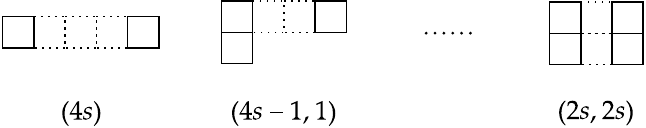}
\caption{All possible Young diagrams in Equation~\eqref{eq:square}.}
\label{fig:square}
\end{figure}

On the other hand, our difference formula in the case $L=2$ gives 
\begin{equation}
    \left[\chi_{(2s)}(x)\right]^2 \:=\: \sum_{\lambda} \mu_{\lambda} \chi_{\lambda}(x)\,,  
\end{equation}
with $\mu_{\lambda} = \mathcal{D}_{R_{A_r}(t)} c_{s,2}(\vec{M})$ for arbitrary $A_r$. Therefore, we can conclude the following proposition, combining Pieri's formula and our difference formula.
\begin{prop} When $L=2$, the combinatorial coefficients $c_{s,2}(\vec{M})$ in \eqref{eq:genexpan} satisfy the following equations,
\begin{equation}
    \mathcal{D}_{R_{A_r}(t)} c_{s,2}(\vec{M}) \:=\: \left\{ \begin{array}{ll}
    1, &  \lambda(\vec{M}) \in\left\{ (4s), (4s-1,1), \dots, (2s,2s)\right\} \\
    0, & \text{otherwise}
    \end{array}\,,
    \right.
\end{equation}
where $\lambda(\vec{M})$ is determined by \eqref{eq:map}, namely $\lambda(\vec{M})_i = M_{i-1} - M_i $.
\end{prop}

More generally, we can consider applying Pieri's formula to the case $\nu = (2s^\prime)$ and, without loss of generality, suppose that $s^\prime \geq s$. Similarly, one can conclude the following
\begin{prop}
    In the case of $L=2$, the coefficient $c_{\vec{s}}(\vec{M})$ in \eqref{eq:genexpan2} shall satisfy 
\begin{equation}
    \mathcal{D}_{R_{A_r}(t)} c_{\vec{s}}(\vec{M}) \:=\: \left\{ \begin{array}{ll}
    1, &  \lambda(\vec{M}) \in\left\{ (2s^\prime+2s), (2s^\prime+2s-1,1), \dots, (2s^\prime,2s)\right\} \\
    0, & \text{otherwise}
    \end{array}\,,
    \right.
\end{equation}
where $\lambda(\vec{M})$ is the same as in the previous proposition.
\end{prop}

\subsection{Kostka numbers}
Let us first recall the definition of Kostka numbers and some of their properties.
\begin{defn}
    The Kostka number $K_{\nu\lambda}$ is the number of semistandard tableaux on $\nu$ of type $\lambda$, namely, the number of ways one can fill the boxes of the Young diagram of $\nu$ with $\lambda_1$ $1$'s, $\lambda_2$ $2$'s, up to $\lambda_k$ $k$'s, following the rule of Young tableaux.
\end{defn}

The Kostka numbers can appear as coefficients \cite[Equation~(A.9)]{fulton1991representation}
\begin{equation}
    H_{\nu}(x) \:=\: \sum_{\lambda} K_{\lambda\nu} S_{\lambda}(x)\,,
\end{equation}
where $K_{\lambda\nu}$ are the Kostka numbers on $\lambda$ of type $\nu$, $S_{\lambda}(x)$ Schur polynomials and $H_{\nu}(x)$ complete symmetric polynomials associated with the Young diagram $\nu$. Alternatively, one can rewrite the above equation as below
\begin{equation}
\label{eq:kostka1}
    K_{\lambda \nu} \:=\: [\Delta\cdot H_{\nu}]_{(\lambda_1 +k-1,\lambda_2 +k-2,\dots,\lambda_k)}\,,
\end{equation}
where $\Delta$ is the Vandermonde determinant.

Given a symmetric polynomial $P(x)$, it is known that it can be expanded in terms of monomial symmetric polynomials $\mathcal{M}_{\lambda}(x)$ or in terms of Schur polynomials $S_{\nu}(x)$,
\[
\begin{aligned}
    P(x) &\:=\: \sum_{\lambda} \mathcal{M}_{\lambda}(x) [P]_{\lambda}\,,\\
    P(x) &\:=\: \sum_{\nu} S_{\nu}(x) [\Delta \cdot P]_{(\nu_1 +k-1,\nu_2 +k-2,\dots,\nu_k)}\,,
\end{aligned}
\]
where $[P]_{\lambda}$ and $[\Delta \cdot P]_{(\nu_1 +k-1,\nu_2 +k-2,\dots,\nu_k)}$ are expansion coefficients of ${x_1}^{\lambda_1}\cdots {x_{k}}^{k}$ and ${x_1}^{\nu_1+k-1} {x_2}^{\nu_2+k-2}\cdots {x_k}^{\nu_k}$ in $P(x)$, respectively. The Kostka numbers can be viewed as coefficients in the change of bases for symmetric polynomials:
\begin{thm}[Lemma~A.26 in \cite{fulton1991representation}] \label{thm:kostka}
For any symmetric polynomial $P$, the following identity holds
\begin{equation}
\label{eq:kostka}
    [P]_{\lambda} \:=\: \sum_{\nu} K_{\nu\lambda} [\Delta \cdot P]_{(\nu_1 +k-1,\nu_2 +k-2,\dots,\nu_k)}\,.
\end{equation}
\end{thm}

For our purposes, we take $k=r+1$ and $P(x) = \left[ \chi_{(2s)}(x) \right]^{L}=:H_{\delta}(x)$, which is by definition a symmetric polynomial of degree $2sL$. Here, $H_{\delta}(x)$ is the complete symmetric polynomial associated with the rectangle Young diagram $\delta = (2s, 2s, \dots, 2s)$ of $L$ rows and $2s$ columns. Then, one can immediately recognize that the coefficients $[H_\delta(x)]_{\lambda}$ and $[\Delta \cdot H_\delta(x)]_{(\nu_1 +r,\nu_2 +r-1,\dots,\nu_{r+1})}$ are just $d_{s,L}(\lambda)$ (or equivalently $c_{s,L}(\vec{M}(\lambda))$) and $\mu_{\nu}$, respectively. Therefore, Equation~\eqref{eq:kostka} leads to
\[
d_{s,L}(\lambda) \:=\: \sum_{\nu} K_{\nu\lambda}\  \mu_\nu\,,
\]
and, in this sense, Equation~\eqref{eq:proofdiff2} in the proof provides a reciprocal version of the above equation. The inverse Kostka numbers can be read directly from \eqref{eq:proofdiff2}.

The above discussions apply to the more general case with $P(x) = H_{\delta}(x)$, where $H_{\delta}(x)$ is determined by the Young diagram $\delta = (2s_1,\dots,2s_L)$, with the coefficient $d_{s,L}(\lambda)$ changed to $d_{\vec{s}}(\lambda)$.

\subsection{Frobenius character formula and the hook length formula}
The Frobenius character formula, or simply the Frobenius formula, provides a method to compute characters of irreducible representations of symmetric groups. The Frobenius formula can be expressed as the following:
\begin{thm}[Frobenius formula]
    \begin{equation}
    \label{eq:fcf}
        \prod_{j} P_j(x)^{i_j} \:=\: \sum_{\lambda} \chi_{\lambda}(C_{\bold{i}}) S_{\lambda}(x)\,,
    \end{equation}
    where $P_j(x)$ is the power sum $P_j(x):= \sum_{i=1}^{k}{x_i}^j$, $C_{\bold{i}}$ denotes the conjugacy class in the symmetric group $\mathcal{S}_{d}$ with the index $\bold{i}=(i_1,\dots,i_d)$ meaning $C_{\bold{i}}$ consists $i_\ell$ $\ell$-cycles for $\ell=1,\dots,d$ and $\sum_{\alpha=1}^{d} \ell i_{\ell} = d$, and $\chi_{\lambda}(C_{\bold{i}})$ is the character of the representation $V_\lambda$ on the conjugacy class $C_{\bold{i}}$ of $\mathcal{S}_{d}$. $S_{\lambda}(x)$ is the associated Schur polynomial.
\end{thm}
Below we take $k=r+1$ and $d=L$. Taking the conjugacy class $C_{\bold{i}}$ with conjugacy structure ${\bold{i}}=(L)$, {\it i.e.} consisting of $L$ cycles of length $1$, ${\rm dim}(V_\lambda):=\chi_{\lambda}(C_{(L)})$ is the dimension of $\lambda$-representation of $\mathcal{S}_{L}$ and the Forbenius's formula \eqref{eq:fcf} tells us
\begin{equation}
\label{eq:ff1}
    \left( x_1 + \dots + x_{r+1}\right)^{L} \:=\: \sum_{\lambda} {\rm dim}(V_\lambda) S_{\lambda}(x)\,,
\end{equation}
where ${\rm dim}(V_\lambda)$ is the number of standard Young tableaux associated with the Young diagram $\lambda$. One can compute $[\Delta(x)\cdot(x_1+\cdots+x_{r+1})^L]_{(\lambda_1+r,\lambda_2+r-1,\dots,\lambda_{r+1})}$ and obtain the hook-length formula:
\begin{prop}[Hook-length formula]
\label{prop:hklength}
\begin{equation}
\label{eq:hklength}
    {\rm dim}(V_\lambda) \:=\: \frac{L!}{\prod_{i,j} h_{\lambda}(i,j)}\,.
\end{equation}
\end{prop}
With the identification of $S_{\lambda}(x)$ and $\chi_{\lambda}(x)$ as they are both Schur polynomials, Equation~\eqref{eq:ff1} can be viewed as a special case of Equation~\eqref{eq:schurexpan} when $s=1/2$. Therefore,  the coefficients on the right hand side of both equations should equal, namely ${\rm dim}(V_\lambda) = \mu_{\mu}$, which implies another formula for ${\rm dim}(V_\lambda)$ using our difference formula \eqref{eq:thm}:
\begin{prop}
\label{prop:hkdiff}
\begin{equation}
\label{eq:hkdiff}
    {\rm dim}(V_\lambda) \:=\: \mu_{\lambda} \:=\: \mathcal{D}_{R_{A_r}(t)} c_{1/2,L}(\vec{M})\,.
\end{equation}
\end{prop}
\begin{remark}
We have provided another proof of Proposition~\ref{prop:hkdiff} in \cite{Shu:2024crv} by directly calculating \eqref{eq:hkdiff}, which reproduces the hook length formula \eqref{eq:hklength}.
\end{remark}

\vspace{8mm}

\section{Branching rule for tensor product representation}
\label{sec:branch}
In this section, we generalize Theorem~\ref{thm:main} and consider decomposition of the tensor-power representation into representations restricted to subalgebras of $A_r$. Consider the root space decomposition of a general semisimple Lie algebra
\begin{equation}
    \mathfrak{g} \:=\: \mathfrak{g}_0 \bigoplus \left(\bigoplus_{\alpha \in \Phi}\mathfrak{g}_{\alpha} \right)\,,
\end{equation}
where $\mathfrak{g}_0$ is its Cartan subalgebra. For each root $\alpha\in\Phi$ one can associate a $A_1$ subalgebra, with generators $H_{\alpha}$, $E_{\alpha}$ and $E_{-\alpha}$. Therefore, a subalgebra of $\mathfrak{g}$ can be studied by investigating how multiple $A_1$'s ``enhance'' to this subalgebra. This is equivalent to looking into how a subset of positive roots generates the root systems of the subalgebra.

We will now focus on describing the subalgebras of $\mathfrak{g}=A_r$ by a subset of positive roots as follows. Denote by $\Phi^+$ the positive roots of $A_r$. The simple roots of $A_r$ can be written in terms of the standard basis,$\{L_i\}_{i=1}^{r+1}$ as $\alpha_i = L_i - L_{i+1}$ for $i=1,\dots,r$, and any positive root $\beta \in \Phi^+$ can be expressed as $\beta = L_i - L_j$ for some $i < j$.

The positive roots ${\Psi}^+$ associated with a sugalgebra $\mathfrak{h} \subset A_r$ is a subset of positive roots of $A_r$, ${\Psi}^+ \subset \Phi^+$ , which can be generically written as a disjoint union of $k$ components:
\[
{\Psi}^+ \:=\: \Phi^+_{1} \sqcup \cdots \sqcup \Phi^+_{k} \subset \Phi^+\,,
\]
where each component $\Phi^+_{i}$ generates the roots system of $A_{r_i}$ and thus determines the subalgebra structure. In terms of standard basis, the above disjoint union corresponds to $k$ disjoint index sets $I_1,\dots, I_k \subset \{1,2,\dots,r+1\}$, such that any element of $\Phi_j^+$ can be written as $L_{j_m} - L_{j_n}$ for some $j_m<j_n \in I_j$. Therefore, a maximum rank subalgebra $\mathfrak{h}$ of $A_r$, determined by this subset ${\Psi}^+$, is given as
\begin{equation}
\label{eq:subalg}
    \mathfrak{h}(\Psi^+) \::=\: {\mathfrak{gl}(1)}^{\oplus n}\oplus A_{r_1} \oplus \cdots \oplus A_{r_k} \subset A_r\,,
\end{equation}
where we have added the possible remaining ${\mathfrak{gl}(1)}$'s to make the rank of $\mathfrak{h}$ maximal and $n:=r-\sum_i r_i$.

The structure of the subset ${\Psi}^+ = \bigsqcup_{j=1}^k \Phi_j^+ \subset \Phi^+$ naturally induces the embedding of the subalgebra, $\mathfrak{h} \hookrightarrow A_r$. We will keep the same labeling for roots in $\Psi^+$ as in $\Phi^+$, which can keep track of subalgebra structures. Even though two subalgebras might be isomorphic, the decompositions onto restricted representations could be different due to different embeddings.

\begin{remark}
As we have discussed in \cite{Shu:2024crv}, various choices of subset ${\Psi}^+$ correspond to various twisted boundary conditions. More specifically, for each element $\alpha = \sum_i \alpha_i$ in ${\Psi}^+$, the linear combination of  $\sum_i \theta_i$ vanishes, where $\theta_i$'s correspond to parameters of twisted boundary conditions. In particular, if a simple root $\alpha_i$ appears in ${\Psi}^+$, then $\theta_i = 0$. 
\end{remark}

\begin{example}[label=ex:decomp]
Let us consider two subalgebras of $\mathfrak{g}=A_r$. 
\begin{itemize}[leftmargin=11pt]
    \item Take ${\Psi}^+ = \{\alpha_1+\alpha_2, \alpha_3, \alpha_1+\alpha_2 + \alpha_3, \alpha_5\}$, ${\Psi}^+$ can be written as a disjoint union of two components:
    \[
        {\Psi}^+ \:=\: \{\alpha_1+\alpha_2, \alpha_3, \alpha_1+\alpha_2 + \alpha_3\} \sqcup \{\alpha_5\}=: \Phi_1^+ \sqcup \Phi_2^+\,,
    \]
    and it is easy to see that $\Phi_1^+$ and $\Phi_2^+$ correspond to positive roots of $\mathfrak{sl}(3)$ and $\mathfrak{sl}(2)$, respectively. Then, the associated subalgebra of consideration is
    \[
        \mathfrak{h}(\Psi^+)\:=\:\mathfrak{sl}(3)\oplus \mathfrak{sl}(2) \oplus \mathfrak{gl}(1)^{r-3} \subset   \mathfrak{sl}(r+1)  \,.
    \]
    \item If we choose a subset of positive roots $\{\alpha_1+\alpha_2, \alpha_3, \alpha_1+\alpha_2 + \alpha_3, \alpha_4\}$, the structure of positive roots in $A_r$ automatically enhances this set to a minimal set of positive roots that contains $\{\alpha_1+\alpha_2, \alpha_3, \alpha_1+\alpha_2 + \alpha_3, \alpha_4\}$:
    \[
        {\Psi}^+ \::=\: \{\alpha_1+\alpha_2, \alpha_3, \alpha_1+\alpha_2 + \alpha_3, \alpha_4, \alpha_3+\alpha_4, \alpha_1+\alpha_2+\alpha_3+\alpha_4\}\,,
    \]
    and the subalgebra $\mathfrak{h}({\Psi}^+)$ is $\mathfrak{sl}(4)\oplus \mathfrak{gl}(1)^{\oplus r-3}$.
\end{itemize}
\end{example}

We want to study the decomposition of the tensor product representation into restricted representations in this subalgebra $\mathfrak{h}$. Therefore, instead of considering \eqref{eq:schurexpan}, we need to consider the following expansion
\begin{equation}
\label{eq:branchexpan}
   \left[ \chi_{(2s)}(A_r)\right]^L \:=\: \sum_{\lambda} \mu_{\lambda}^{\mathfrak{h}(\Psi^+)} \chi_{\lambda}(\mathfrak{h}(\Psi^+))\,,
\end{equation}
where $\lambda$ labels irreducible representations for subalgebra $\mathfrak{h}(\Psi^+)$ and $\lambda$ should be written in several components depending on the structure of this subalgebra. Suppose that the subalgebra $\mathfrak{h}(\Psi^+)$ takes the form of \eqref{eq:subalg}, then $\lambda=\{\lambda_1,\dots,\lambda_{r+1}\}$ can be written in components of two types:
\begin{itemize}[leftmargin=11pt]
    \item for each $i=1,\dots,k$, denote the indices in $I_{i}$ by $i_{1}, \dots, i_{r_i+1}$, and then $\lambda^{(i)}=(\lambda_{i_1},\dots, \lambda_{i_{r_{i+1}}})$ labels irreducible representations of $A_{r_i+1}$, in the sense of Remark~\ref{rmk:red} and ~\ref{rmk:order};
    \item for $m\in \hat{I}$ with $\hat{I}:=\{1,\dots,r+1\}- \cup_{j=1}^k I_j$, $\lambda_m$ labels irreducible representations, or equivalently charges, of $\mathfrak{gl}(1)$'s in \eqref{eq:subalg}.
\end{itemize}
As the summation on the right-hand side of \eqref{eq:branchexpan} is over the restricted representations of $\mathfrak{h}\subset A_r$, one can directly obtain $\lambda$ from the Young diagrams of $A_r$ by looking at the subset $\Psi^+$. That is, for each disjoint component $\Phi_i \subset \Psi^+$, rows labeled by the associated index set $I_i$ can be extracted from the Young diagrams of $A_r$ and obtain a smaller Young diagram associated with $\lambda^{(i)}$, which shall label the irreducible representations of $A_{r_i}$. Since the index sets $I_i\cap I_j = \varnothing$ for every $i\neq j$, $\lambda^{(i)}$ and $\lambda^{(j)}$ do not share common rows from the Young diagrams of $A_r$. After extracting all $A_{r_i}$'s, the unpicked rows in the Young diagrams of $A_r$ correspond to the charges of the remaining $\mathfrak{gl}(1)$'s.

The Weyl character of the subalgebra $\chi_{\lambda}(\mathfrak{h}(\Psi^+))$ in Equation~\eqref{eq:branchexpan} is a product of the Weyl characters of each component, using $(r+1)$ formal variables $x_i$ as before,
\[
    \chi_{\lambda}(\mathfrak{h}(\Psi^+)) \:=\: \prod_{m\in \hat{I}} {x_{m}}^{\lambda_m}\prod_{i=1}^k \chi_{\lambda^{(i)}}(A_{r_i})\,.
\]

As discussed in the introduction, Theorem~\ref{thm:branching} provides a difference formula to calculate the multiplicity $\mu_{\lambda}^{\mathfrak{h}(\Psi^+)}$, from the expansion coefficients $c_{s,L}(\vec{M})$ as in \eqref{eq:genexpan}. Again, $\lambda_i$'s and $\vec{M}$ are related by the dictionary \eqref{eq:map}. Let us provide a proof below.
\begin{proof}[Proof of Theorem~\ref{thm:branching}]
    Note that the nontrivial Weyl denominator in $\chi_{\lambda}(\mathfrak{h}(\Psi^+))$ gets contributions from nonabelian factors $A_{r_i}$'s, which can be written as a product of Vandermonde determinants:
    \begin{equation}
        \Delta(\mathfrak{h}) \:=\: \prod_{i=1}^k \Delta^{(i)}(x)\,,
    \end{equation}
    where
    \begin{equation}
        \Delta^{(i)}(x) \:=\: \prod_{\substack{ i_a, i_b \in I_i \\i_a < i_b} } (x_{i_a} - x_{i_b})\,.
    \end{equation}
    Similarly to the proof of Theorem~\ref{thm:main1}, we need to extract the coefficient in front of the monomial 
    \[
        \prod_{m\in \hat{I}} {x_{m}}^{\lambda_m} \prod_{i=1}^k \prod_{j=1}^{r_i+1} {x_{i_j}}^{\lambda_{i_j} + r_{i}+1-j}
    \]
    in the product 
    \begin{equation}
        \prod_{i=1}^k \Delta^{(i)}(x) \cdot \left[ \chi_{(2s)}(x)\right]^L \,.
    \end{equation}
    This can be done by recursively acting $\Delta^{(i)}(x)$ on $\left[ \chi_{(2s)}(x)\right]^L$ for $i=1,\dots k$, because $I_i$'s are disjoint sets of indices and each $\Delta^{(i)}(x)$ can be viewed as an operator mapping symmetric functions to antisymmetric functions, leading to a factorization property at each step. Finally, we only need to match the factor $\prod_{m\in \hat{I}} {x_{m}}^{\lambda_m}$. This leads to a similar formula as \eqref{eq:thm2},
    \begin{equation}
    \label{eq:thm3}
        \mu^{\mathfrak{h}(\Psi^+)}_{\lambda} \:=\: \sum_{\sigma\in W(\mathfrak{h})} {\rm sgn}(\sigma) c_{s,L}(M_j - \sum_{i=1}^{j}\left(\sigma(i) - i\right)) \, ,
    \end{equation}
    where $W(\mathfrak{h}) = S_{r_1+1}\times \cdots \times S_{r_k+1}$ is the Weyl group of the sub Lie algebra $\mathfrak{h}$. 
    
    Each of the steps above leads to an action of the shift operator $\mathcal{D}_{R_{r_i}(t^{(i)})}$, where $t^{(i)}$ stands for the $t$-variables from $\{x_{i_j} \mid \text{for}\ i_j\in I_i\}$, by the same definition as \eqref{eq:t-x}. Due to the factorization property (Lemma~\ref{lem:factor}), we will end up with
    \[
        \mu_{\lambda}^{\mathfrak{h}(\Psi^+)} \:=\: \prod_{i=1}^k\mathcal{D}_{R_{r_i}(t^{(i)})} c_{s,L}(\vec{M}) =  \mathcal{D}_{R_{\mathfrak{h}(\Psi^+)}(t)} c_{s,L}(\vec{M})\,,
    \]
    where $\lambda$ and $\vec{M}$ are related by \eqref{eq:map}.
\end{proof}

\begin{example}[continues=ex:decomp]
Let us just look at the first case with $\Psi^+ = \Phi_1^+ \sqcup \Phi_2^+$, where $ \Phi_1^+ = \{ \alpha_1 + \alpha_2, \alpha_3, \alpha_1+\alpha_2+\alpha_3\} $ and $ \Phi_2^+ \{\alpha_5\}$. The summation on the left-hand side of \eqref{eq:branching} should be over the smaller Young diagrams as illustrated in Figure~\ref{fig:branching} together with remaining rows that are unpicked.
\begin{figure}[!h]
\centering
\includegraphics[scale=0.65]{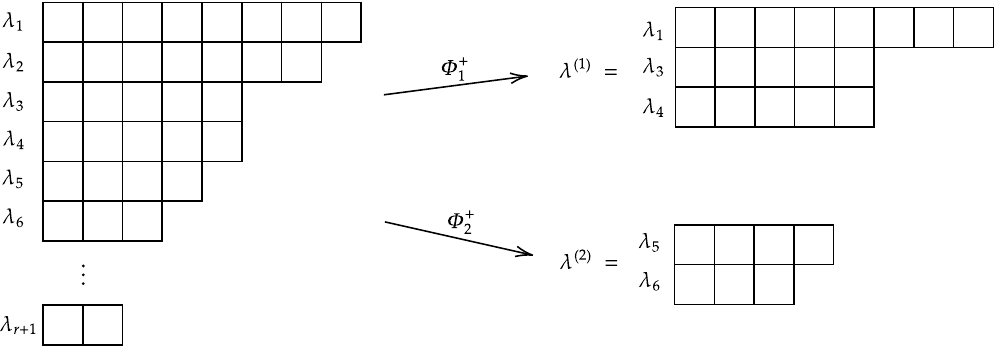}
\caption{An illustration of obtaining Young diagrams for subalgebras.}
\label{fig:branching}
\end{figure}

Then, in this case,
\[
    R_{\mathfrak{h}(\Psi^+)}(t) \:=\: \frac{1}{(1-t_1 t_2)(1-t_3)(1-t_1 t_2 t_3)} \frac{1}{(1-t_5)}\,,
\]
where the first factor corresponds to the Weyl character of the subalgebra component $\mathfrak{sl}(3)$ and the second one corresponds to the Weyl character of $\mathfrak{sl}(2)$. Then the multiplicity $\mu_{\lambda}^{\mathfrak{h}(\Psi^+)}$ is obtained by acting the shift operator associated with the above $R_{\mathfrak{h}(\Psi^+)}(t)$ on $c_{s,L}(\vec{M})$, where $M_i$ and $\lambda_i$ are related by \eqref{eq:map}.
\end{example}

The generalization to the case of the tensor product of different spins is straightforward. With the relation between $\lambda_i$ and $M_i$ given in \eqref{eq:map2}, we have
\begin{cor}\label{cor:branching}
\[
    \mu_{\lambda}^{\mathfrak{h}(\Psi^+)} \:=\: \mathcal{D}_{R_{\mathfrak{h}(\Psi^+)}(t)} c_{\vec{s}}(\vec{M})\,.
\]
\end{cor}

\vspace{8mm}

\section{Conjecture for A-type Lie superalgebras}\label{sec:super} 
Let us consider the A-type Lie superalgebra case as explored in \cite{Shu:2022vpk,Shu:2024crv}. Take $\mathfrak{g} = \mathfrak{sl}(m|n)$, and its positive roots $\Phi^{+}$ consist of two parts, $\Phi^+ = \Phi_0^+ \cup \Phi_1^+$, where $\Phi_0^+$ and $\Phi_1^+$ are the set of positive even roots and odd roots, respectively. The positive roots can be written in the standard basis \cite{VanderJeugt:1989ak,frappat:hal-00376660},
\[
\begin{aligned}
    & \Phi_0^+ \:=\: \left\{L_i - L_j | i<j,\ i,j = 1,\ldots,m \right\} \cup \left\{K_k - K_\ell | k<\ell,\ k,\ell = 1,\ldots,n  \right\}\, , \\
    & \Phi_1^+ \:=\: \left\{L_i - K_k | i = 1,\ldots,m,\ k = 1,\ldots, n \right\}\,,
\end{aligned}
\]
and they can be generated by the {\it Distinguished Simple Root System}
\[
\alpha_i \:=\: L_i - L_{i+1}\,,  \quad \beta_\ell \:=\: K_\ell - K_{\ell+1}\,, \quad  \delta \:=\: L_{m} - K_1\,,
\]
for $i=1,\ldots,m-1$ and $\ell = 1,\ldots, n-1$. 

\begin{figure}[!h]
    \centering
    \includegraphics[width=0.35\linewidth]{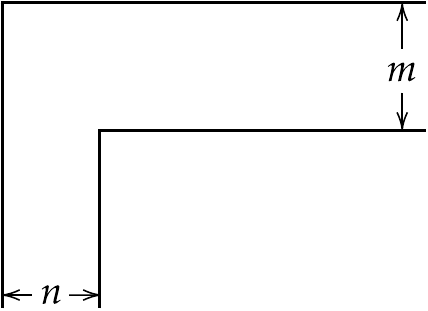}
    \caption{$(m,n)$-hook, constraining the Young diagrams for irreducible covariant tensor representations of $\mathfrak{sl}(m|n)$.}
    \label{fig:hook}
\end{figure}

In the following, we will only look at the covariant tensor representations $V_{\lambda}$ of $\mathfrak{sl}(m|n)$. $V_{\lambda}$ is an irreducible representation specified by a {\it Hook Young Diagram} $\lambda$ contained in the $(m,n)$-hook as shown in Figure~\ref{fig:hook}, and the character of $V_\lambda$ is given by the {\it Hook-Schur Function} due to \cite{Berele:1987yi}.
\begin{thm}[Berele-Regev]
    The character of $V_{\lambda}$ is given by the hook-Schur function $\HS_{\lambda}(x;y)$, which is defined by 
    \[
        \HS_{\lambda}(x;y) \:=\: \sum_{\tau<\lambda} S_{\lambda/\tau}(x) S_{\tau^\prime}(y)\,,
    \]
    where $x$ is short for $(x_1,\dots,x_m)$ and $y=(y_1,\dots,y_n)$ with $x_i$ and $y_\ell$ the formal variables defined as
    \[
        x_i \:=\: e^{L_i},\quad y_\ell \:=\: e^{K_\ell}\,.
    \]
    In the abvoe, $S_{\lambda}(x)$ is the Schur function, $\tau^\prime$ is the conjugate of $\tau$.
\end{thm}

As in the ordinary A-type Lie algebra case, we are interested in considering the decomposition of a tensor power of $V_{(\ell)}$, namely,
\begin{equation}
    V_{(2s)}^{\otimes L} \:=\: \oplus_{\lambda}\,\mu_{\lambda}\, V_{\lambda}\,.
\end{equation}
In terms of character, 
\begin{equation}
\label{eq:superdecomp}
    \left[\HS_{(2s)}(x;y)\right]^{L} \:=\: \sum_{\lambda}\, \mu_{\lambda}\, \HS_{\lambda}(x;y)\,.
\end{equation}
Our difference formula can be directly generalized to this Lie superalgebra case by introducing the shift operator associated with an analogy of the Weyl denominator,
\begin{equation}
\label{eq:superdeno}
    R_{\mathfrak{sl}(m|n)}(t) \:=\: \frac{\prod_{\alpha\in \Phi_0^+} (1-t^{\alpha})}{\prod_{\alpha\in \Phi_1^+ }(1+t^{\alpha})}\,.
\end{equation}
where $t^\alpha :=  \big(\prod_{i=1}^{m-1} t_i^{m_i }\big) t_{m}^{k}\big(\prod_{\ell=1}^{n-1} t_{\ell+m}^{n_{\ell}}\big) $ for a general positive root 
\[
\alpha \:=\: \sum_{i=1}^{m-1} m_i \alpha_i + k \delta + \sum_{\ell=1}^{n-1} n_\ell \beta_\ell\,,
\]
and the formal variables $\{t_I\}_{I=1}^{m+n+1}$ are defined associating with each simple root as below (a straightforward generalization of the ordinary case as discussed before),
\[
    t_i \:=\: e^{-\alpha_i}\,,\quad t_m \:=\: e^{-\delta}\,, \quad t_{\ell+m} \:=\: e^{-\beta_\ell}\,,
\]
for $i=1,\dots,m-1$ and $\ell =1,\dots, n-1$.

In addition to the expansion by irreducible modules $V_\lambda$ as \eqref{eq:superdecomp}, one can also expand
\begin{equation}
\label{eq:superexpan}
    \left[\HS_{(2s)}(x;y)\right]^{L} \:=\: \sum_{\vec{M}} c_{s,L}(\vec{M}) \prod_{i=1}^m x_i^{M_{i-1}-M_i} \prod_{\ell=1}^n y_\ell^{M_{m+\ell-1}-M_{m+\ell}}\,,
\end{equation}
with $M_0 = 2sL$ and $M_{m+n} = 0$. As stated in Conjecture~\ref{conj:super}, applying the shifted operator associated with $R_{\mathfrak{sl}(m|n)}(t)$ to the coefficients $c_{s,L}(\vec{M})$ will give us the tensor product multiplicities $\mu_{\lambda}$ in \eqref{eq:superdecomp}, where $\lambda$ and $\vec{M}$ are related by \eqref{eq:mapsuper}. In the following, we give a proof for the special case of $\mathfrak{sl}(1|1)$ and provide some numerical evidence for cases of $\mathfrak{sl}(2|1)$ and $\mathfrak{sl}(1|2)$.

\subsection{Example - $\mathfrak{sl}(1|1)$} The root system of $\mathfrak{sl}(1|1)$ is rather simple, which only consists of two odd roots, and the simple root can be taken as $\delta = L - K$.

   Consider the Lie superalgebra $\mathfrak{sl}(1|1)$ and the tensor power of the representation $(2s)$, corresponding to the one-row Young diagram of $2s$ boxes. When $s=1/2$, it becomes the fundamental representation. The character of the representation $(2s)$ is
   \[
    \HS_{(2s)}(x;y) \:=\: x^{2s}+x^{2s-1} y\,.
   \]
   The $L$-th power of $\HS_{(2s)}(x;y)$ has the expansion
   \[
   \left[\HS_{(2s)}(x;y)\right]^L \:=\: \sum_{M=0}^{L}c_{s,L}(M)\, x^{2sL-M} y^M\,,
   \]
   with $c_{s,L}(M) = \binom{L}{M}$, which is independent of $s$. While the expansion over the irreducible covariant tensor modules $V_{\lambda}$ is given as
   \[
   \left[\HS_{(2s)}(x;y)\right]^L \:=\: \sum_{\lambda}\mu_{\lambda}\, \HS_{\lambda}(x;y)\,,
   \]
   with $\lambda$ contained in the $(1,1)$-hook, as shown in Figure~\ref{fig:hookdiagram}. For a given generic $\lambda(s,M):= (2sL-M, 1, \dots, 1)$ with $M$ 1's and $2sL>M$, or equivalently $\lambda_1 = 2sL-M$ and $\lambda_2^\prime = M$, one can straightforwardly find that
   \[
    \HS_{(2sL-M,1,\dots,1)}(x;y) \:=\: x^{2sL-M} y^M + x^{2sL-M-1} y^{M+1}\,.
   \]
   \begin{figure}[t!]
   \begin{minipage}{0.3\textwidth}
    \centering
    \includegraphics[width=0.9\linewidth]{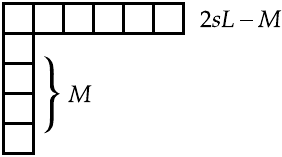}
   \end{minipage}
   \hfill
   \begin{minipage}{0.3\textwidth}
    \centering
    \includegraphics[width=0.9\linewidth]{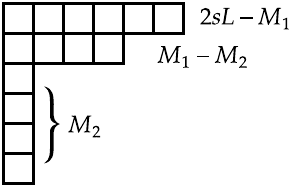}
   \end{minipage}
   \hfill
   \begin{minipage}{0.3\textwidth}
    \centering
    \includegraphics[width=0.9\linewidth]{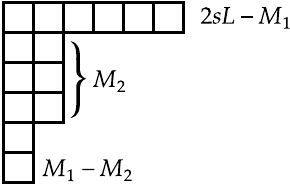}
   \end{minipage}
   \caption{Examples of Young diagrams in the $(1,1)$-hook, $(2,1)$-hook and $(1,2)$-hook.}
   \label{fig:hookdiagram}
   \end{figure}
    Comparing two expansions, one shall obtain that
    \[
        \binom{L}{M} \:=\: c_{s,L}(M) \:=\: \mu_{\lambda(s,M)} + \mu_{\lambda(s,M-1)}\,,
    \]
    with initial condition $\mu_{\lambda(s,M=0)} = 1$, and therefore one can conclude
    \[
        \mu_{\lambda(s, M)} \:=\: \sum_{i=0}^{\infty}(-1)^i\binom{L}{M-i}\,,
    \]
    and we adopt the convention that $\binom{L}{M} = 0$ for $M>L$ and $M<0$.
    
    This is consistent with our conjectured formula. By our definition,
    \[
    R_{\mathfrak{sl}(1|1)} \:=\: \frac{1}{1+t} \:=\: \sum_{i=0}^{\infty} (-1)^i t^i\,,
    \]
    and so the associated shifted operator leads to 
    \[
    \mathcal{D}_{R_{\mathfrak{sl}(1|1)}(t)} c_{s,L}(M) \:=\: \sum_{i=0}^{\infty} (-1)^i c_{s,L}(M-i) \:=\: \sum_{i=0}^{\infty}(-1)^i\binom{L}{M-i}\,.
    \]

    In this sense, we can conclude that the above derivation provides a proof for our conjectured formula in the case of $\mathfrak{sl}(1|1)$.
    
\subsection{Example - $\mathfrak{sl}(2|1)$}
Let us look at a slightly more nontrivial example $\mathfrak{sl}(2|1)$ and compare its results with those of $\mathfrak{sl}(1|2)$. The distinguished simple roots for $\mathfrak{sl}(2|1)$ can be taken as
    \[
        \alpha \:=\: L_1 - L_2, \quad \delta = L_2 - K\,.
    \]

In the representation $\lambda = (2s)$, the hook-Shur function for $\mathfrak{sl}(2|1)$ is given by
    \[
        \HS_{(2s)}(x_1,x_2;y) \:=\: S_{(2s)}(x_1,x_2) + y S_{(2s-1)}(x_1,x_2)\,,
    \]
    where $S_{\lambda}(x_1,x_2)$ is the Schur function. In particular, when $s=1/2$, 
    \[
        \HS_{(1)}(x_1,x_2;y) \:=\: x_1 + x_2 + y\,.
    \]
    The $L$-th tensor power of $\HS_{(1)}(x_1,x_2;y)$ can be written as
    \[
    \begin{aligned}
        \left[\HS_{(1)}(x_1,x_2;y)\right]^L  &\:=\: \sum_{0\leq M_2 \leq M_1 \leq L} c_{1/2,L}(M_1,M_2)\, x_1^{L-M_1}x_2^{M_1 -M_2} y^{M_2}\,. \\
        &\:=\: \sum_{\lambda} \mu_\lambda\, \HS_{\lambda}(x_1,x_2;y) \\
    \end{aligned}
    \]
    In the above equations, it is straightforward to find out that
    \begin{equation}
    \label{eq:21ExpCoeff}
        c_{1/2,L}(M_1,M_2) \:=\: \binom{L}{M_1}\binom{M_1}{M_2}\,,
    \end{equation}
    and $\lambda$ in the second summation are $(2,1)$-Young diagrams with total $2sL$ boxes as shown in Figure~\ref{fig:hookdiagram}. Note that
    \begin{equation}\label{eq:sl21factor}
    \begin{aligned}
        R_{\mathfrak{sl}(2|1)} &\:=\: \frac{1-t_1}{(1+t_2)(1+t_1 t_2)}\\
        &\:=\: \sum_{k,\ell=0}^{\infty} (-1)^{k+\ell} t_1^k t_2^{k+\ell} - \sum_{k,\ell=0}^{\infty} (-1)^{k+\ell} t_1^{k+1} t_2^{k+\ell} \\
        &\:=\: \sum_{\ell=0}^{\infty} (-1)^\ell t_2^{\ell} - \sum_{\ell=0}^{\infty} (-1)^{\ell} t_1^{\ell+1} t_2^{\ell}\,,
    \end{aligned}
    \end{equation}
    where the third equation is due to many cancelations in the second line, then the associated shifted operator acting on $c_{1/2,L}(M_1,M_2)$ gives 
    \begin{equation}
    \label{eq:21diff}
    \begin{aligned}
        &\mathcal{D}_{R_{\mathfrak{sl}(2|1)}} c_{1/2,L}(M_1,M_2) \\
        &\:=\: \sum_{\ell=0}^{\infty} (-1)^\ell \binom{L}{M_1}\binom{M_1}{M_2 - \ell} - \sum_{\ell=0}^{\infty}(-1)^{\ell} \binom{L}{M_1-\ell-1}\binom{M_1-\ell-1}{M_2 - \ell}\\
        &\:=\: \binom{L}{M_1} \binom{M_1 - 1}{M_2} - \binom{L}{M_1-M_2-1} \binom{L-M_1 + M_2}{M_2}
    \end{aligned}
    \end{equation}
    where in the last line we have used the finite-binomial identity 
    \[
        \sum_{k=0}^{r} (-1)^k \binom{n}{k} \:=\: (-1)^r \binom{n-1}{r} \,,
    \]
    for $0\leq r\leq n-1$ and the binomial identity 
    \[
        \binom{n}{k}\binom{k}{r} \:=\: \binom{n}{r}\binom{n-r}{k-r}\,.
    \]
    
    Let us examine the specific case where $L=6$. Our conjectured formula, applied in this case as Equation~\eqref{eq:21diff}, gives 
    {
    \begin{center}
    \renewcommand{\arraystretch}{1.2}
    \begin{tabular}{c | c | c} 
    \toprule
        $(M_1,M_2)$ & $\lambda=(2sL-M_1,M_1-M_2,\overbrace{1,\dots,1}^{M_2})$ & $\mathcal{D}_{R_{\mathfrak{sl}(2|1)}} c_{1/2,L}(M_1,M_2)$\\ \hline
        $(0,0)$     &$(6)$      & $1$ \\ 
        $(1,0)$     &$(5,1)$      & $5$ \\
        $(2,0)$     &$(4,2)$      & $9$ \\
        $(2,1)$     &$(4,1,1)$      & $10$ \\
        $(3,0)$     &$(3,3)$      & $5$ \\
        $(3,1)$     &$(3,2,1)$      & $16$ \\
        $(3,2)$     &$(3,1,1,1)$    & $10$ \\
        $(4,2)$     &$(2,2,1,1)$    & $9$ \\
        $(4,3)$     &$(2,1,1,1,1)$  & $5$ \\
        $(5,4)$     &$(1,1,1,1,1,1)$& $1$ \\
    \bottomrule
    \end{tabular}
    \end{center}
    }
    
    The most left column in the above table, which is obtained from the conjectured formula, agrees with the expansion coefficients $\mu_{\lambda}$ in hook-Schur functions
    \begin{equation}\label{eq:21expan}
    \begin{aligned}
        & \left[\HS_{(1)}(x_1,x_2;y)\right]^6 \\
        & \quad \:=\: \HS_{(6)}(x_1,x_2;y) + 5 \HS_{(5,1)}(x_1,x_2;y) + 9 \HS_{(4,2)}(x_1,x_2;y) \\
        & \qquad + 10 \HS_{(4,1,1)}(x_1,x_2;y) + 5 \HS_{(3,3)}(x_1,x_2;y) + 16 \HS_{(3,2,1)}(x_1,x_2;y) \\
        & \qquad + 10 \HS_{(3,1,1,1)}(x_1,x_2;y) + 9 \HS_{(2,2,1,1)}(x_1,x_2;y)  \\
        & \qquad + 5 \HS_{(2,1,1,1,1)}(x_1,x_2;y) + \HS_{(1,1,1,1,1,1)}(x_1,x_2;y)\,.
    \end{aligned}
    \end{equation}
    
    \subsection{Example - $\mathfrak{sl}(1|2)$} Note that the Lie superalgebra $\mathfrak{sl}(1|2)$ is isomorphic to $\mathfrak{sl}(2|1)$. This fact is also reflected in the decomposition of their tensor power representation. In the case of $\mathfrak{sl}(1|2)$, $\HS_{(1)}(x;y_1,y_2) = x + y_1 + y_2$ and so
    \[
        \left[\HS_{(1)}(x;y_1,y_2)\right]^L \:=\: \sum_{0\leq M_2 \leq M_1 \leq L} c_{1/2,L}(M_1,M_2) x^{L-M_1} y_1^{M_1-M_2} y_2^{M_2}\,.
    \]
    and the coefficients $c_{1/2,L}(M_1,M_2)$ share the same expression as \eqref{eq:21ExpCoeff}. However, in this case, the shift operator $\mathcal{D}_{R_{\mathfrak{sl}(1|2)}}$ we should apply is determined by
    \begin{equation}
    \begin{aligned}
        R_{\mathfrak{sl}(1|2)} \:=\: \frac{1-t_2}{(1+t_1)(1+t_1 t_2)} = \sum_{\ell=0}^\infty (-1)^\ell t_1^{\ell} - \sum_{\ell=0}^{\infty} (-1)^\ell t_1^{\ell} t_2^{\ell+1}\,.
    \end{aligned}
    \end{equation}
    and our conjectured difference formula gives
    \begin{equation}
    \label{eq:12diff}
    \begin{aligned}
        &\mathcal{D}_{R_{\mathfrak{sl}(1|2)}} c_{1/2,L}(M_1,M_2) \\
        &\:=\: \sum_{\ell=0}^{\infty} (-1)^\ell \binom{L}{M_1-\ell}\binom{M_1-\ell}{M_2} - \sum_{\ell=0}^{\infty}(-1)^{\ell} \binom{L}{M_1-\ell}\binom{M_1-\ell}{M_2 - \ell-1}\\
        &\:=\: \binom{L}{M_2} \binom{L-M_2-1}{M_1-M_2} - \binom{L}{M_1-M_2+1} \binom{L-M_1 + M_2-2}{M_2-1} \,.
    \end{aligned}
    \end{equation}

    Let $L=6$. The Young diagram $\lambda$ is related to $(L,M_1,M_2)$ as in the setup of Conjecture~\ref{conj:super}. More specifically,
    \[
        \lambda_1 \:=\: L-M_1,\quad, \tilde{\lambda}_1 \:=\: M_1 - M_2, \quad \tilde{\lambda}_2 \:=\: M_2\,.
    \]
    Then, the conjectured formula gives
    \begin{center}
    \renewcommand{\arraystretch}{1.2}
    \begin{tabular}{c | c | c}
    \toprule
        $(M_1,M_2)$ & $\lambda$ & $\mathcal{D}_{R_{\mathfrak{sl}(1|2)}} c_{1/2,L}(M_1,M_2)$\\ \hline
        $(0,0)$     &$(6)$          & $1$ \\ 
        $(1,0)$     &$(5,1)$        & $5$ \\
        $(2,0)$     &$(4,1,1)$      & $10$ \\
        $(2,1)$     &$(4,2)$        & $9$ \\
        $(3,0)$     &$(3,1,1,1)$    & $10$ \\
        $(3,1)$     &$(3,2,1)$      & $16$ \\
        $(4,0)$     &$(2,1,1,1,1)$  & $5$ \\
        $(4,1)$     &$(2,2,1,1)$    & $9$ \\
        $(4,2)$     &$(2,2,2)$      & $5$ \\
        $(6,6)$     &$(1,1,1,1,1,1)$& $1$ \\
    \bottomrule
    \end{tabular}
    \end{center}
    and this table is consistent with the expansion in terms of hook-Schur functions,
    \begin{equation}\label{eq:12expan}
    \begin{aligned}
        &\left[\HS_{(1)}(x;y_1,y_2)\right]^6 \\
        &\quad \:=\: \HS_{(6)}(x;y_1,y_2) + 5 \HS_{(5,1)}(x;y_1,y_2) + 9 \HS_{(4,2)}(x;y_1,y_2) \\
        & \qquad + 10 \HS_{(4,1,1)}(x;y_1,y_2) + 16 \HS_{(3,2,1)}(x;y_1,y_2) + 10 \HS_{(3,1,1,1)}(x;y_1,y_2) \\
        & \qquad + 5 \HS_{(2,1,1,1,1)}(x;y_1,y_2) + 9 \HS_{(2,2,1,1)}(x;y_1,y_2) + 5 \HS_{(2,2,2)}(x;y_1,y_2) \\
        & \qquad + \HS_{(1,1,1,1,1,1)}(x;y_1,y_2)\,.
    \end{aligned}
    \end{equation}

    \vspace{4mm}
    
    \noindent \textbf{$\mathfrak{sl}(2|1)$~v.s.~$\mathfrak{sl}(1|2)$.} As an aside, one can compare the expansions \eqref{eq:21expan} and \eqref{eq:12expan}, which are related by the conjugation of Young diagrams. Recall the identity of the hook-Schur functions associated with mutual conjugate Young diagrams. Suppose that $\lambda$ is a $(m,n)$-Young diagram and $\lambda^\prime$ is its conjugation, then
    \begin{thm}[Theorem~6.13 in \cite{Berele:1987yi}]
        \[
            \HS_{\lambda}(x_1,\dots,x_m;y_1,\dots,y_n) \:=\: \HS_{\lambda^\prime}(y_1,\dots,y_n;x_1,\dots,x_m)\,.
        \]
    \end{thm}
    It is then straightforward to conclude that the two expansions \eqref{eq:21expan} and \eqref{eq:12expan} are related by sending $(x_1, x_2)$ to $(y_1,y_2)$, $y$ to $x$ and $(2,1)$-diagrams to $(1,2)$-diagrams by conjugation. This identification of the decomposition of tensor powers echoes the isomorphism between $\mathfrak{sl}(2|1)$ and $\mathfrak{sl}(1|2)$.

    \subsection{Branching rule} The Corollary~\ref{cor:branching} can also be generalized in the Lie superalgebra case to be Conjecture~\ref{conj:superbranching}. If we consider a subalgebra $\mathfrak{h} \subset \mathfrak{sl}(m|n)$, and, as before, $\mathfrak{h}$ can be specified by a subset of positive roots, $\Psi^+ \subset \Phi^+(\mathfrak{sl}(m|n))$, which in general can be written as the union of even and odd roots as $\Psi^+ = \Psi_0^+\sqcup \Psi_1^+$. We will also consider expansion \eqref{eq:branchexpan}, but change $A_r$ to $\mathfrak{sl}(m|n)$, and compute multiplicities $\mu_\lambda^{\mathfrak{h}(\Psi^+)}$. Correspondingly, we shall introduce
    \begin{equation}
    \label{eq:superdenosub}
        R_{\mathfrak{h}(\Psi^+)}(t) \:=\: \frac{\prod_{\alpha\in \Psi_0^+} (1-t^{\alpha})}{\prod_{\alpha\in \Psi_1^+ }(1+t^{\alpha})}\,.
    \end{equation}
    Then, under the relation between $\lambda_i$'s and $M_i$'s as \eqref{eq:mapsuper}, we further conjecture that the multiplicity $\mu_\lambda^{\mathfrak{h}(\Psi^+)}$ can be computed from the coefficients $c_{s,L}(\vec{M})$ in \eqref{eq:superexpan} by the formula in Conjecture~\ref{conj:superbranching}, namely by applying the shifted operator associated with $R_{\mathfrak{h}(\Psi^+)}(t)$ as defined in \eqref{eq:superdenosub}.

    In the following, we investigate Conjecture~\ref{conj:superbranching} by looking at the subalgebras of $\mathfrak{sl}(2|1)$ and provide a proof of the conjecture in the special case of $\mathfrak{sl}(m)\subset \mathfrak{sl}(m|1)$.

    \subsubsection{Subalgebras of $\mathfrak{sl}(2|1)$} The Lie superalgebra $\mathfrak{sl}(2|1)$ has three positive roots 
    \[
    \{\alpha \:=\: L_1-L_2, \delta \:=\: L_2 - K, \beta \:=\: L_1 - K\}\,,
    \]
    and it has three nontrivial subalgebras, $\mathfrak{sl}(2)$, $\mathfrak{sl}(1|1)_1$ and $\mathfrak{sl}(1|1)_2$, which are determined by taking $\Psi^+ = \{\alpha\},\ \{\delta\}$ and $\{\beta\}$, respectively. Here, we have used the subscript $1$ or $2$ to distinguish two different embeddings of $\mathfrak{sl}(1|1) \hookrightarrow \mathfrak{sl}(2|1)$, and this point is more clear from the choice of positive roots.
    
    \vspace{2mm}
    
    \paragraph{$\mathbf{\mathfrak{sl}(2|1) \supset \mathfrak{sl}(2)}$} This subalgebra $\mathfrak{sl}(2)$ of $\mathfrak{sl}(2|1)$ is obtained by choosing the subset of positive roots  $\Psi^+=\{\alpha\}$. Correspondingly, we shall take
        \[
            R_{\mathfrak{sl}(2)(\alpha)}(t) \:=\: 1 - t_1\,,
        \]
        which, according to Conjecture~\ref{conj:superbranching}, tells us that the multiplicities can be calculated as the following
        \[
        \begin{aligned}
        \mu_{\lambda}^{\mathfrak{sl}(2)(\alpha)} &\:=\: \mathcal{D}_{R_{\mathfrak{sl}(2)(\alpha)}(t)} c_{s,L}(M_1,M_2) \\
        &\:=\: c_{s,L}(M_1,M_2) - c_{s,L}(M_1-1,M_2) \,,    
        \end{aligned}
        \]
        where $\lambda = (L-M_1, M_1 - M_2)$. More explicitly, below, we have listed the results for $L=6$ and $s=1/2$, in which case $c_{1/2,L}(M_1,M_2) = \binom{L}{M_1} \binom{M_1}{M_2}$:
        \begin{center}
        \renewcommand{\arraystretch}{1.2}
        \begin{longtable}{c | c | c} 
        \toprule
        $(M_1,M_2)$ & $\lambda = (L-M_1,M_1-M_2)$ & $c_{1/2,L}(M_1,M_2) - c_{1/2,L}(M_1-1,M_2)$\\ \hline
        $(0,0)$     &$(6)$          & $1$ \\ 
        $(1,0)$     &$(5,1)$        & $5$ \\
        $(1,1)$     &$(5)$          & $6$ \\
        $(2,0)$     &$(4,2)$        & $9$ \\
        $(2,1)$     &$(4,1)$        & $24$ \\
        $(2,2)$     &$(4)$        & $15$ \\
        $(3,0)$     &$(3,3)$        & $5$ \\
        $(3,1)$     &$(3,2)$        & $30$ \\
        $(3,2)$     &$(3,1)$        & $45$ \\
        $(3,3)$     &$(3)$        & $20$ \\
        $(4,2)$     &$(2,2)$        & $30$ \\
        $(4,3)$     &$(2,1)$        & $40$ \\
        $(4,4)$     &$(2)$        & $15$ \\
        $(5,4)$     &$(1,1)$        & $15$ \\
        $(5,5)$     &$(1)$        & $6$ \\
        $(6,6)$     &$(0)$        & $1$ \\
        \bottomrule
        \end{longtable}
        \end{center}
        
        The above results are consistent with the expansion coefficients in hook-Schur and Schur functions,
        \begin{equation}\label{eq:s21s12exp}
            \begin{aligned}
                &\left[ \HS_{(1)}(x_1,x_2;y)\right]^6\\ 
                &\quad \:=\: S_{(6)}(x_1,x_2) + 5 S_{(5,1)}(x_1,x_2)+6 y S_{(5)}(x_1,x_2) + 9 S_{(4,2)}(x_1,x_2)\\
                &\qquad + 24 y S_{(4,1)}(x_1,x_2) + 15 y^2 S_{(4)}(x_1,x_2) + 5 S_{(3,3)}(x_1,x_2) \\
                &\qquad + 30 y S_{(3,2)}(x_1,x_2) + 45 y^2 S_{(3,1)}(x_1,x_2) + 20 y^3 S_{(3)}(x_1,x_2)\\
                &\qquad + 30 y^2 S_{(2,2)}(x_1,x_2) + 40 y^3 S_{(2,1)}(x_1,x_2)+ 15 y^4 S_{(2)}(x_1,x_2) \\
                &\qquad + 15 y^4 S_{(1,1)}(x_1,x_2) + 6 y^5 S_{(1)}(x_1,x_2) + y^6\,.
            \end{aligned}
        \end{equation}

        \vspace{2mm}
        
        \paragraph{$\mathbf{\mathfrak{sl}(2|1) \supset \mathfrak{sl}(1|1)_1}$} As introduced earlier, the subalgebra $\mathfrak{sl}(1|1)_1$ is determined by taking $\Psi^+ = \{\delta\}$, and so 
        \[
            R_{\mathfrak{sl}(1|1)_1(\delta)}(t) \:=\: \frac{1}{1+t_2} = \sum_{\ell=0}^\infty (-1)^{\ell} t_2^{\ell}\,.
        \]
        The associated shifted operator leads to 
        \[
        \begin{aligned}
        \mu_{\lambda}^{\mathfrak{sl}(1|1)_1(\delta)} &\:=\: \mathcal{D}_{R_{\mathfrak{sl}(1|1)_1(\delta)}(t)} c_{s,L}(M_1,M_2) \\
        &\:=\: \sum_{\ell=0}^{\infty }(-1)^{\ell}c_{s,L}(M_1,M_2-\ell)\,,  
        \end{aligned}
        \]
        The results for the case of $L=6$ and $s=1/2$ are summarized in the table below, where $\lambda = (0)$ means the trivial representation.
        \begin{center}
        \renewcommand{\arraystretch}{1.2}
        \begin{longtable}{c | c | c}
        \toprule
        $(M_1,M_2)$ & $\lambda = (M_1-M_2,\overbrace{1,\dots,1}^{M_2})$ & $\sum_{\ell=0}^{\infty }(-1)^{\ell}c_{1/2,L}(M_1,M_2-\ell)$\\ \hline
        $(0,0)$     &$(0)$          & $1$ \\ 
        $(1,0)$     &$(1)$          & $6$ \\ 
        $(2,0)$     &$(2)$          & $15$ \\ 
        $(2,1)$     &$(1,1)$        & $15$ \\ 
        $(3,0)$     &$(3)$          & $20$ \\ 
        $(3,1)$     &$(2,1)$        & $40$ \\ 
        $(3,2)$     &$(1,1,1)$      & $20$ \\ 
        $(4,0)$     &$(4)$          & $15$ \\ 
        $(4,1)$     &$(3,1)$        & $45$ \\ 
        $(4,2)$     &$(2,1,1)$      & $45$ \\ 
        $(4,3)$     &$(1,1,1,1)$    & $15$ \\ 
        $(5,0)$     &$(5)$          & $6$ \\ 
        $(5,1)$     &$(4,1)$        & $24$ \\ 
        $(5,2)$     &$(3,1,1)$      & $36$ \\ 
        $(5,3)$     &$(2,1,1,1)$    & $24$ \\ 
        $(5,4)$     &$(1,1,1,1,1)$  & $6$ \\ 
        $(6,0)$     &$(6)$          & $1$ \\ 
        $(6,1)$     &$(5,1)$        & $5$ \\ 
        $(6,2)$     &$(4,1,1)$      & $10$ \\ 
        $(6,3)$     &$(3,1,1,1)$    & $10$ \\ 
        $(6,4)$     &$(2,1,1,1,1)$  & $5$ \\ 
        $(6,5)$     &$(1,1,1,1,1,1)$& $1$ \\ 
        \bottomrule
        \end{longtable}
        \end{center}
        
        One can directly check that the above results are consistent with the tensor power of hook-Schur functions and its decomposition as the following.
        \begin{equation}\label{eq:s21s111exp}
            \begin{aligned}
                &\left[\HS_{(1)}(x_1,x_2;y)\right]^6 \\
                &\quad \:=\: x_1^6 + 6 x_1^5 \HS_{(1)}(x_2;y) + 15 x_1^4 \HS_{(2)}(x_2;y) + 15 x_1^4 \HS_{(1,1)}(x_2;y)  \\
                &\qquad + 20 x_1^3 \HS_{(3)}(x_2;y) + 40 x_1^3 \HS_{(2,1)}(x_2;y) + 20 x_1^3 \HS_{(1,1,1)}(x_2;y) \\
                &\qquad + 15 x_1^2 \HS_{(4)}(x_2;y) + 45 x_1^2 \HS_{(3,1)}(x_2;y) + 45 x_1^2 \HS_{(2,1,1)}(x_2;y) \\
                &\qquad + 15 x_1^2 \HS_{(1,1,1,1)}(x_2;y) + 6 x_1 \HS_{(5)}(x_2;y) + 24 x_1 \HS_{(4,1)}(x_2;y) \\
                &\qquad + 36 x_1 \HS_{(3,1,1)}(x_2;y) + 24 x_1 \HS_{(2,1,1,1)}(x_2;y) + 6 x_1 \HS_{(1,1,1,1,1)}(x_2;y) \\
                &\qquad + \HS_{(6)}(x_2;y) + 5 \HS_{(5,1)}(x_2;y) + 10 \HS_{(4,1,1)}(x_2;y) \\
                &\qquad + 10 \HS_{(3,1,1,1)}(x_2;y) + 5\HS_{(2,1,1,1,1)}(x_2;y) + \HS_{(1,1,1,1,1,1)}(x_2;y)\,.
            \end{aligned}
        \end{equation}

        \vspace{2mm}
        
        \paragraph{$\mathbf{\mathfrak{sl}(2|1) \supset \mathfrak{sl}(1|1)_2}$} A similar study applies to $\mathfrak{sl}(1|1)_2$ with $\Psi^+ = \{\beta\}$ and
        \[
            R_{\mathfrak{sl}(1|1)_2(\beta)}(t) \:=\: \frac{1}{1+t_1 t_2} = \sum_{\ell=0}^{\infty}(-1)^\ell t_1^{\ell}t_2^{\ell}\,,
        \]
        Therefore, our conjectured formula gives
        \[
        \begin{aligned}
        \mu_{\lambda}^{\mathfrak{sl}(1|1)_2(\beta)} &\:=\: \mathcal{D}_{R_{\mathfrak{sl}(1|1)_2(\beta)}(t)} c_{s,L}(M_1,M_2) \\
        &\:=\: \sum_{\ell=0}^{\infty }(-1)^{\ell}c_{s,L}(M_1-\ell,M_2-\ell)\,,  
        \end{aligned}
        \]
        and we have also listed results in the special case with $L=6$ and $s=1/2$ in the table below, which again agree with results directly computed by using the hook-Schur functions as shown in \eqref{eq:s21s112exp}.
        \begin{center}
        \renewcommand{\arraystretch}{1.2}
        \begin{longtable}{c | c | c}
        \toprule
        $(M_1,M_2)$ & $\lambda = (L-M_1,\overbrace{1,\dots,1}^{M_2})$ & $\sum_{\ell=0}^{\infty }(-1)^{\ell}c_{s,L}(M_1-\ell,M_2-\ell)$\\ \hline
        $(0,0)$     &$(6)$          & $1$ \\ 
        $(1,0)$     &$(5)$          & $6$ \\ 
        $(1,1)$     &$(5,1)$        & $5$ \\ 
        $(2,0)$     &$(4)$          & $15$ \\ 
        $(2,1)$     &$(4,1)$        & $24$ \\ 
        $(2,2)$     &$(4,1,1)$      & $10$ \\ 
        $(3,0)$     &$(3)$          & $20$ \\ 
        $(3,1)$     &$(3,1)$        & $45$ \\ 
        $(3,2)$     &$(3,1,1)$      & $36$ \\ 
        $(3,3)$     &$(3,1,1,1)$    & $10$ \\ 
        $(4,0)$     &$(2)$          & $15$ \\ 
        $(4,1)$     &$(2,1)$        & $40$ \\ 
        $(4,2)$     &$(2,1,1)$      & $45$ \\ 
        $(4,3)$     &$(2,1,1,1)$    & $24$ \\ 
        $(4,4)$     &$(2,1,1,1,1)$  & $5$ \\ 
        $(5,0)$     &$(1)$          & $6$ \\ 
        $(5,1)$     &$(1,1)$        & $15$ \\ 
        $(5,2)$     &$(1,1,1)$      & $20$ \\ 
        $(5,3)$     &$(1,1,1,1)$    & $15$ \\ 
        $(5,4)$     &$(1,1,1,1,1)$  & $6$ \\ 
        $(5,5)$     &$(1,1,1,1,1,1)$& $1$ \\ 
        $(6,0)$     &$(0)$          & $1$ \\ 
        \bottomrule
        \end{longtable}
        \end{center}
        \begin{equation}\label{eq:s21s112exp}
            \begin{aligned}
                &\left[\HS_{(1)}(x_1,x_2;y)\right]^6 \\
                &\quad \:=\: \HS_{(6)}(x_1;y) + 6 x_2 \HS_{(5)}(x_1;y) + 5 \HS_{(5,1)}(x_1;y) + 15 x_2^2 \HS_{(4)}(x_1;y)\\
                &\qquad + 24 x_2 \HS_{(4,1)}(x_1;y) + 10\HS_{(4,1,1)}(x_1;y) + 20 x_2^3 \HS_{(3)}(x_1;y)\\
                &\qquad + 45 x_2^2 \HS_{(3,1)}(x_1;y) + 36 x_2 \HS_{(3,1,1)}(x_1;y) + 10 \HS_{(3,1,1,1)}(x_1;y) \\
                &\qquad +15 x_2^4 \HS_{(2)}(x_1;y) +40 x_2^3 \HS_{(2,1)}(x_1;y) + 45 x_2^2 \HS_{(2,1,1)}(x_1;y)\\
                &\qquad + 24 x_2 \HS_{(2,1,1,1)}(x_1;y) + 5 \HS_{(2,1,1,1,1)}(x_1;y)+ 6 x_2^5 \HS_{(1)}(x_1;y) \\
                &\qquad + 15 x_2^4 \HS_{(1,1)}(x_1;y) + 20 x_2^3 \HS_{(1,1,1)}(x_1;y) + 15 x_2^2 \HS_{(1,1,1,1)}(x_1;y) \\
                &\qquad + 6 x_2 \HS_{(1,1,1,1,1)}(x_1;y)+ \HS_{(1,1,1,1,1,1)}(x_1;y) + x_2^6\,.
            \end{aligned}
        \end{equation} 
    \begin{remark}
    Similar to the discussion in Section~\ref{sec:branch}, the powers of $y$ in the expansion of hook-Schur functions, \eqref{eq:s21s12exp}, should be understood as charges of the complementary $\mathfrak{gl}(1)$ to $\mathfrak{sl}(2)$ in $\mathfrak{sl}(2|1)$. Similarly,  the power of $x_1$ (resp.~$x_2$) in \eqref{eq:s21s111exp} (resp.~\eqref{eq:s21s112exp}) should also be understood as charges of the complementary $\mathfrak{gl}(1)$ to $\mathfrak{sl}(1|1)_1$ (resp.~$\mathfrak{sl}(1|1)_2$) in $\mathfrak{sl}(2|1)$.
    \end{remark}

    \subsubsection{$\mathfrak{sl}(m|1)\supset \mathfrak{sl}(m)$} We shall consider the $L$-th tensor power of $(2s)$-representations of $\mathfrak{sl}(m|1)$ for $s\in \mathbb{Z}+\frac{1}{2}$ and its decomposition over covariant tensor modules $V_{\lambda}$. Here, $\lambda$ are the hook-Young diagrams contained in the $(m,1)$-hook. The character of $(2s)$-representation can be expressed by the $(2s)$-hook-Schur function for $\mathfrak{sl}(m|1)$, which in this case can be explicitly written out as 
    \begin{equation}
        \HS_{(2s)}(x_1,\dots,x_m;y) \:=\: S_{(2s)}(x_1,\dots,x_m) + y S_{(2s-1)}(x_1,\dots,x_m)\,.
    \end{equation}
    We will use the shorthanded notation $x=(x_1,\dots,x_m)$. Then, consider
    \begin{equation}
        \left[ \HS_{(2s)}(x;y) \right]^L \:=\: \sum_{M_m=0}^L \binom{L}{M_m} y^{M_m} S_{2s}(x)^{L-M_m}(x)S_{2s-1}(x)^{M_m}\,,
    \end{equation}
    and one can realize that for each $M_m$ the product $S_{2s}(x)^{L-M_m}(x)S_{2s-1}(x)^{M_m}$ can be understood as the tensor product of $(L-M_m)$ spin-$s$ and $M_m$ spin-$(s-\frac{1}{2})$, which is just a special case of the tensor product of $L$ different spins discussed in Section~\ref{sec:diffspin}. For each $0\leq M_m \leq L$, let us denote the expansion coefficients by $c_{s,L,M_m}(M_1,\dots,M_{m-1})$ in 
    \begin{multline}
        S_{2s}(x)^{L-M_m}(x)S_{2s-1}(x)^{M_m} \:=\:\\
        \sum_{M_m\leq M_{m-1}\leq \cdots\leq M_1 \leq L} c_{s,L,M_m}(M_1,\dots,M_{m-1}) \, x_1^{2sL-M_1} x_2^{M_1-M_2}\cdots x_m^{M_{m-1}-M_m}\,.
    \end{multline}
    Putting together, we shall have
    \begin{multline}
        \left[ \HS_{(2s)}(x;y) \right]^L \\
        \:=\: \sum_{0\leq M_{m}\leq \cdots\leq M_1 \leq L}  c_{s,L}(\vec{M})\, x_1^{2sL-M_1} x_2^{M_1-M_2}\cdots x_m^{M_{m-1}-M_m} y^{M_m}\,,
    \end{multline}
    with 
    \[
        c_{s,L}(\vec{M}) \:=\: \binom{L}{M_m}c_{s,L,M_m}(M_1,\dots,M_{m-1})\,.
    \]
    
    We want to decompose this $L$-th tensor power of $(2s)$-representation of $\mathfrak{sl}(m|1)$ over the irreducible representations of the subalgebra $\mathfrak{sl}(m) \subset \mathfrak{sl}(m|1)$, namely
    \begin{equation}
        \left[ \HS_{(2s)}(x;y) \right]^L \:=\: \sum_{M_m=0}^L \sum_{\lambda} \mu_{\lambda}^{\mathfrak{sl}(m)} y^{M_m} S_{\lambda}(x)\,,
    \end{equation}
    where $\lambda$ are standard (redundant) Young diagrams of total $2sL-M_m$ boxes.
    
    To obtain such subalgebra $\mathfrak{sl}(m)$, we should choose the subset of positive roots $\Phi^+$ consisting of only $L_i$'s,
    \[
        \Psi^+ \:=\: \langle L_1 - L_2, L_2 - L_3, \dots, L_{m-1} - L_{m} \rangle\,.
    \]
    Then, $R_{sl(m)(\Psi^+)}(t)$ is equal to $R_{sl(m)}(t)$ associated to the ordinary $\mathfrak{sl}(m)$,
    \[
        R_{sl(m)(\Psi^+)}(t) \:=\: \prod_{\alpha\in\Psi^+}(1+t^{\alpha})\,.
    \]
    
    We claim that $\mu_{\lambda}^{\mathfrak{sl}(m)} $ can be obtained by the following formula, which is a specialized case of Conjecture~\ref{conj:superbranching}:
    \begin{prop}
        \begin{equation}
            \mu_{\lambda}^{\mathfrak{sl}(m)} \:=\: \mathcal{D}_{R_{sl(m)(\Psi^+)}(t) } c_{s,L}(\vec{M})\,.
        \end{equation}
    \end{prop}
    \begin{proof}
        Note that $c_{s,L}(\vec{M}) = \binom{L}{M_m}c_{s,L,M_m}(M_1,\dots,M_{m-1})$ and the shift operator $\mathcal{D}_{R_{sl(m)(\Psi^+)}(t) }$ acts only on $c_{s,L,M_m}(M_1,\dots,M_{m-1})$. Then, for a given $M_m$, one can directly apply Corollary~\ref{cor:diffspins} and get the result.
    \end{proof}
    
\vspace{8mm}

\appendix

\section{Nested restricted occupancy problem}
\label{app:nrop}
Let us first recall the nested restricted occupancy problem discussed in \cite{Shu:2024crv}. Let $L \ge M_1 \ge  \ldots \ge M_r \ge 0$ be a sequence of non-increasing integers. The restricted occupancy coefficient $c_{s,L}(\vec M)$ counts the ways of assigning $M_a$ boxes into $L$ cells $\{n^{(a)}_{1}, \ldots, n^{(a)}_{L}\}$, such that $\sum_{\alpha=1}^L n_\alpha^{(a)} = M_a$ and $n_\alpha^{(a)} \le n_\alpha^{(a-1)}$, $n_\alpha^{(1)}\leq s$. This is illustrated by the left graph in Figure~\ref{fig:3d}, and this reduces to the restricted occupancy problem considered in \cite{Freund:1956ro} when $r=1$.
\begin{figure}[!h]
\begin{minipage}{0.48\textwidth}
            \centering
            \includegraphics[width=0.8\linewidth]{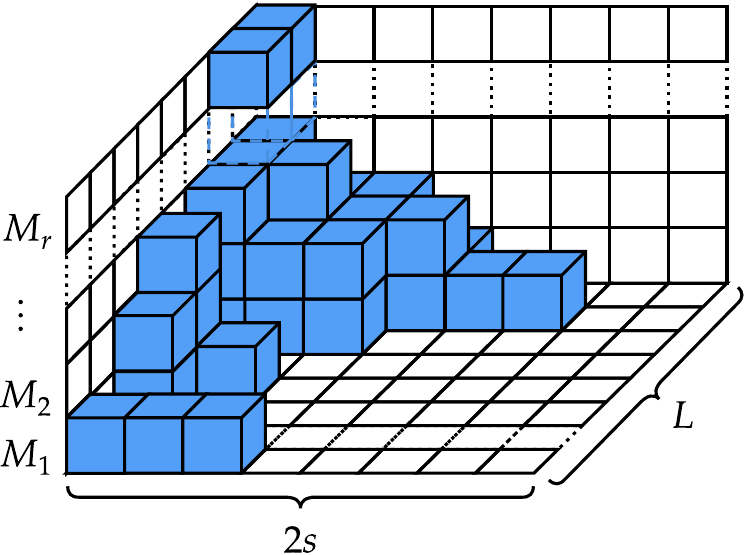}
        \end{minipage}\hfill
        \begin{minipage}{0.48\textwidth}
            \centering
            \includegraphics[width=0.8\linewidth]{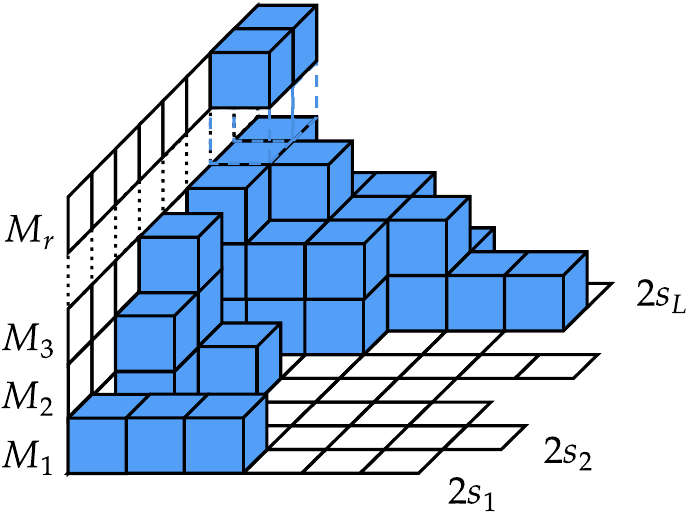}
        \end{minipage}
        \caption{A 3d restricted occupancy problem. On each level we have a 2d restricted occupancy with $n_\alpha^{(a)}$ boxes in each row and a total of $M_a$ boxes. Each row on a higher level has no more boxes than the lower level.}
        \label{fig:3d}
\end{figure}

A slight generalization of this problem is illustrated by the right graph in Figure~\ref{fig:3d}, with a modification of the constraint $n_\alpha^{(1)}\leq s$ to $n_\alpha^{(1)}\leq s_{\alpha}$ for $\alpha =1,\dots, L$.

\section{Identities for $c_{s,L}(\vec{M})$}
\label{app:identity}
The coefficients of $c_{s,L}(\vec{M})$ are defined as the following
\begin{multline}
    \left[S_{(2s)}(x_1,x_2,\dots,x_{r+1})\right]^L \\
    \:=\: \sum_{0\leq M_r \leq \dots \leq M_1 \leq 2sL } c_{s,L}(M_1,\dots,M_r) x_1^{2sL-M_1} x_2^{M_1-M_2} \cdots x_{r}^{M_{r-1}-M_r} x_{r+1}^{M_r} \,.
\end{multline}
Due to the identity of Schur functions, for any permutation $\sigma\in S_{r+1}$,
\begin{equation}
    S_{(2s)}(x_1,x_2,\dots,x_{r+1}) \:=\: S_{(2s)}(x_{\sigma(1)},x_{\sigma(2)},\dots,x_{\sigma(r+1)})\,,
\end{equation}
then we shall have 
\begin{equation}
\label{eq:identity1}
    \begin{aligned}
    &\left[S_{(2s)}(x_1,x_2,\dots,x_{r+1})\right]^L\\
    &\:=\: \sum_{0\leq M_r \leq \dots \leq M_1 \leq 2sL } c_{s,L}(M_1,\dots,M_r) x_1^{2sL-M_1} x_2^{M_1-M_2} \cdots x_{r}^{M_{r-1}-M_r} x_{r+1}^{M_r}\\
    &\:=\: \sum_{0\leq M^\prime_r \leq \dots \leq M^\prime_1 \leq 2sL } c_{s,L}(M^\prime_1,\dots,M^\prime_r) x_{\sigma(1)}^{2sL-M^\prime_1} x_{\sigma(2)}^{M^\prime_1-M^\prime_2} \cdots x_{\sigma(r)}^{M^\prime_{r-1}-M^\prime_r} x_{\sigma(r+1)}^{M^\prime_r}\,.
    \end{aligned}
\end{equation}
One can extract the identities for $c_{s,L}(M_1,\dots,M_r)$'s according to the above equation. 

For example, take 
\[
    \sigma(1,2,3,\dots,r,r+1) \:= \:(2,1,3,\dots,r,r+1)\,,
\]
namely only $1$ and $2$ exchanged, then we shall have
\begin{equation}
    \begin{aligned}
    &\sum_{0\leq M_r \leq \dots \leq M_1 \leq 2sL } c_{s,L}(M_1,\dots,M_r) x_1^{2sL-M_1} x_2^{M_1-M_2}x_3^{M_2-M_3} \cdots x_{r}^{M_{r-1}-M_r} x_{r+1}^{M_r}\\
    &\:=\: \sum_{0\leq M^\prime_r \leq \dots \leq M^\prime_1 \leq 2sL } c_{s,L}(M^\prime_1,\dots,M^\prime_r) x_{2}^{2sL-M^\prime_1} x_{1}^{M^\prime_1-M^\prime_2}x_3^{M^\prime_2-M^\prime_3} \cdots x_{r}^{M^\prime_{r-1}-M^\prime_r} x_{r+1}^{M^\prime_r}\,.
    \end{aligned}
\end{equation}
Now, if take $M_1^\prime = 2sL+M_2-M_1$ and $M_i^\prime = M_i$ for all $i\geq 2$, one can conclude that
\begin{equation}
    \begin{aligned}
    &\sum_{0\leq M_r \leq \dots \leq M_1 \leq 2sL } c_{s,L}(M_1,\dots,M_r) x_1^{2sL-M_1} x_2^{M_1-M_2}x_3^{M_2-M_3} \cdots x_{r}^{M_{r-1}-M_r} x_{r+1}^{M_r}\\
    &\:=\: \sum_{0\leq M_r \leq \dots \leq M_1 \leq 2sL } c_{s,L}(2sL+M_2-M_1,M_2,\dots,M_r) x_{2}^{M_1-M_2} x_{1}^{2sL-M_1}x_3^{M_2-M_3} \cdots x_{r}^{M_{r-1}-M_r} x_{r+1}^{M_r}\,,
    \end{aligned}
\end{equation}
where $M_2\leq M_1\leq 2sL$ is equivalent to $M_2\leq 2sL+M_2-M_1\leq 2sL$, the above equation leads to
\[
     c_{s,L}(M_1,\dots,M_r) \:=\: c_{s,L}(2sL+M_2-M_1,M_2,\dots,M_r)\,.
\]
Similarly, for an adjacent transposition 
$$\sigma: (1,\dots,i,i+1,\dots,r+1) \mapsto (1,\dots,i+1,i,\dots,r+1)$$ 
with $2\leq i\leq r$, namely, only exchanging $i$ and $i+1$, one can follow the same procedure and obtain that
\begin{equation}
\label{eq:transport}
    c_{s,L}(M_1,\dots,M_i,\dots,M_r) \:=\: c_{s,L}(M_1,\dots,M_{i-1}+M_{i+1}-M_i,\dots,M_r)\,,
\end{equation}
where we have used the convention that $M_{r+1}=0$ for simplicity of notation. 

In fact, one can obtain $|S_{r+1}| = (r+1)!$ identities in total for $c_{s,L}(M_1,\dots,M_r)$. For later convenience, given a permutation $\sigma \in S_{r+1}$, we denote the identity obtained following the previous procedure as 
\begin{equation}
\label{eq:identity2}
    c_{s,L}(\vec{M}) \:=\: c_{s,L}(\sigma(\vec{M}))\, ,
\end{equation}
where we have used notation $\vec{M}= (M_1,\dots,M_r)$ and $\sigma(\vec{M})$ denotes parameters after $\sigma$-permutation extract from \eqref{eq:identity1}. 

\begin{remark}
    As demonstrated in \cite{Shu:2022vpk}, the coefficient $c_{s,L}(\vec{M})$ gives the Witten index of the dual quiver gauge theory via the Bethe/gauge correspondence. Then, these nontrivial identities among $c_{s,L}(\vec{M})$ provide evidence for the Seiberg-like duality in the gauge theory. In particular, if one chooses $\sigma$ in \eqref{eq:identity2} as a transposition, one can easily check that the identity associated with this $\sigma$ coincides with the quiver mutation in the cluster algebra.
\end{remark}

Identities \eqref{eq:identity2} are consistent with Lemma~\ref{lem:equal}. For the identity \eqref{eq:transport} obtained from an adjacent transposition, it is equivalent to the exchange of $\lambda_i$ and $\lambda_{i+1}$ using \eqref{eq:map}, consistent with symmetric property in the expansion \eqref{eq:monexpan}.

In the general cases of tensor product of different spins, starting with the following expansion,
\begin{multline}
    \prod_{i=1}^{L}\left[S_{(2s_i)}(x_1,x_2,\dots,x_{r+1})\right] \\
    \:=\: \sum_{0\leq M_r \leq \dots \leq M_1 \leq 2|\vec{s}| } c_{\vec{s}}(M_1,\dots,M_r) x_1^{2sL-M_1} x_2^{M_1-M_2} \cdots x_{r}^{M_{r-1}-M_r} x_{r+1}^{M_r} \,,
\end{multline}
and following the same procedures as before, one can also extract the identities for $c_{\vec{s}}(\vec{M})$ due to the symmetric property of $\prod_{i=1}^{L}\left[S_{(2s_i)}(x_1,x_2,\dots,x_{r+1})\right]$.

\section{A brute force derivation for $A_1$ case}
\label{app:a1}
The Weyl character formula for spin-$s$ representation of $\mathfrak{sl}(2)$ gives
\begin{equation}
    \chi_{(2s)}(x) \:=\: \frac{x_1^{2s+1}-x_2^{2s+1}}{x_1-x_2}\, ,
\end{equation}
which is a Schur polynomial of $x := (x_1, x_2)$ determined by the Young diagram $\lambda = (2s)$. It is a degree-$2s$ complete symmetric polynomial with integer coefficients, and therefore its $L$-th power is a homogeneous symmetric polynomial with integer coefficients but of degree $2sL$,
\begin{equation} \label{eq:a1int}
    \left[\chi_{(2s)}(x)\right]^L \:=\: \sum_{M=0}^{2sL} c_{s,L}(M) x_1^{2sL - M} x_2^{M} \,,
\end{equation}
The Weyl character $\chi_{(2s)}(x)$ is invariant under the exchange $x_1 \leftrightarrow x_2$, so does its $L$-th power and one can obtain the following identity, \footnote{Please see appendix~\ref{app:identity} for a derivation for general cases.}
\begin{equation} \label{eq:a1dual}
    c_{s,L}(M) \:=\: c_{s,L}(2sL-M)\,.
\end{equation}

On the other hand, the decomposition of $L$-th tensor power of spin-$s$ representation into a direct sum of irreducible representations can be expressed by equation~\eqref{eq:weyldenom} with $\lambda = (\lambda_1, \lambda_2)$ Young diagrams labelling these irreducible representations and $\chi_{\lambda}(x)$'s in the summation are, written in terms of Schur polynomials,
\begin{equation*}
    \chi_{\lambda}(x) \:=\: \frac{x_1^{\lambda_1+1} x_2^{\lambda_2} - x_1^{\lambda_2} x_2^{\lambda_1+1} }{x_1 - x_2}\, .
\end{equation*}
The main theorem \ref{thm:main} for this case states that 
\begin{equation*}
    \mu_{(2sL-M,M) } \:=\: c_{s,L}(M) - c_{s,L}(M-1)\,,
\end{equation*}
which in this simple case can be directly derived as below.

First, rewrite the RHS of equation~\eqref{eq:a1int} as
\begin{equation*}
    \begin{aligned}
        &\sum_{0\leq M \leq 2sL } c_{s,L}(M) x_1^{2sL-M} x_2^M \\
        &\:=\: \frac{1}{x_1 - x_2} \left( \sum_{0\leq M \leq 2sL } c_{s,L}(M) x_1^{2sL-M+1} x_2^M - \sum_{0\leq M \leq 2sL } c_{s,L}(M) x_1^{2sL-M} x_2^{M+1} \right) \\
        &\:=\: \frac{1}{x_1 - x_2} \left( \sum_{0\leq M \leq 2sL } c_{s,L}(M) x_1^{2sL-M+1} x_2^M - \sum_{1\leq M \leq 2sL+1 } c_{s,L}(M-1) x_1^{2sL-M+1} x_2^{M} \right) \\
        &\:=\: \frac{1}{x_1 - x_2} \left( \sum_{0\leq M \leq sL } c_{s,L}(M) x_1^{2sL-M+1} x_2^M + \sum_{sL+1\leq M \leq 2sL } c_{s,L}(M) x_1^{2sL-M+1} x_2^M \right. \\ & \quad \left. - \sum_{1\leq M \leq sL } c_{s,L}(M-1) x_1^{2sL-M+1} x_2^{M} - \sum_{sL+1\leq M \leq 2sL+1 } c_{s,L}(M-1) x_1^{2sL-M+1} x_2^{M} \right)\, .
    \end{aligned}
\end{equation*}
Due to the identity \eqref{eq:a1dual}, then we shall have
\begin{equation*}
    \begin{aligned}
        \sum_{sL+1\leq M \leq 2sL } c_{s,L}(M) x_1^{2sL-M+1} x_2^M &= \sum_{sL+1\leq M \leq 2sL } c_{s,L}(2sL-M) x_1^{2sL-M+1} x_2^M \\ 
        &\:=\: \sum_{0\leq M \leq sL-1 } c_{s,L}(M) x_1^{M+1} x_2^{2sL-M} \\
        &\:=\: \sum_{1\leq M \leq sL } c_{s,L}(M-1) x_1^{M} x_2^{2sL-M+1}\,,
    \end{aligned}
\end{equation*}
and, similarly, one can obtain
\begin{equation*}
    \sum_{sL+1\leq M \leq 2sL+1 } c_{s,L}(M-1) x_1^{2sL-M+1} x_2^{M} \:=\: \sum_{0\leq M \leq sL } c_{s,L}(M) x_1^{M} x_2^{2sL-M+1}\,.
\end{equation*}
Therefore,
\begin{equation*}
    \begin{aligned}
        \sum_{0\leq M \leq 2sL } c_{s,L}(M) x_1^{2sL-M} x_2^M &\:=\: \frac{1}{x_1 - x_2} \left( \sum_{0\leq M \leq sL } c_{s,L}(M) \left(x_1^{2sL-M+1} x_2^M - x_1^{M} x_2^{2sL-M+1}  \right) \right. \\ 
        & \quad \left. - \sum_{1\leq M \leq sL } c_{s,L}(M-1) \left(x_1^{2sL-M+1} x_2^{M} - x_1^{M} x_2^{2sL-M+1}\right) \right) \\
        & \:=\: \sum_{0\leq M\leq sL} \left(c_{s,L}(M)-c_{s,L}(M-1)\right) \chi_{(2sL-M,M)}\,.
    \end{aligned}
\end{equation*}
Compared with equation~\eqref{eq:weyldenom} by identifying $\lambda = (2sL-M,M)$, one can conclude that $\mu_{(2sL-M,M)} = c_{s,L}(M)-c_{s,L}(M-1)$, and this completes the derivation of the Equation~\eqref{eq:a1thm}.

\bibliographystyle{alpha}
\bibliography{ref}

\end{document}